\documentclass[a4paper, 10pt, reqno]{amsart}

\usepackage{amsthm}
\usepackage{amsmath}
\usepackage{amssymb}
\usepackage{amstext}
\usepackage{amsfonts}
\usepackage[english]{babel}
\usepackage[usenames,dvipsnames]{color}
\usepackage{comment}
\usepackage{enumitem}
\usepackage{fancyhdr}
\usepackage[T1]{fontenc}
\usepackage[height=21cm, left=2.45cm, right=2.45cm]{geometry}
\usepackage[hidelinks, colorlinks=true, urlcolor=blue]{hyperref}
\usepackage{latexsym}
\usepackage{mathrsfs}
\usepackage{mathtools}
\usepackage{stackengine}
\usepackage{stmaryrd}

\newtheorem{theorem}{Theorem}[section]
\newtheorem{lemma}[theorem]{Lemma}
\newtheorem{corollary}[theorem]{Corollary}
\newtheorem{proposition}[theorem]{Proposition}

\theoremstyle{definition}
\newtheorem{example}[theorem]{Example}

\newtheorem{remarks}[theorem]{\bf Remarks}
\newtheorem{claim}{\textsc{Claim}}

\newcommand{\N}{\mathbb N}
\newcommand{\Z}{\mathbb Z}

\newcommand{\Q}{\mathbb Q}

\providecommand{\eps}{\varepsilon}
\providecommand{\LLc}{\mathscr{L}}
\providecommand{\LLs}{\mathsf{L}}
\providecommand{\HHb}{H}
\providecommand{\NNb}{{\mathbb{N}_0}}
\providecommand{\NNp}{\mathbb{N}}

\providecommand{\PPc}{\mathscr{P}}

\providecommand{\RRb}{\mathbb{R}}
\providecommand{\UUc}{\mathscr{U}}
\providecommand{\ZZb}{\mathbb{Z}}

\providecommand\llb{\llbracket}
\providecommand\rrb{\rrbracket}
\providecommand\maxrho{\rho}

\newcommand\pprime{{\prime\prime}}
\renewcommand{\emptyset}{\varnothing}

\newcommand{\red}{{\text{\rm red}}}
\newcommand{\fin}{\text{\rm fin}}
\newcommand{\BF}{\text{\rm BF}}

\DeclareMathOperator{\spec}{spec}

\DeclareMathOperator{\End}{End} \DeclareMathOperator{\Min}{Min}

\newcommand{\DP}{\negthinspace : \negthinspace}

\renewcommand{\t}{\mid}

\numberwithin{equation}{section}

\hypersetup{
    pdftitle={Arithmetic of commutative semigroups},
    pdfauthor={Yushuang Fan, Alfred Geroldinger, Florian Kainrath, and Salvatore Tringali},
    pdfmenubar=false,
    pdffitwindow=true,
    pdfstartview=FitH,
    colorlinks=true,
    linkcolor=blue,
    citecolor=green,
    urlcolor=blue
}

\DeclareMathSymbol{\widehatsym}{\mathord}{largesymbols}{"62}

\providecommand{\eps}{\varepsilon}
\providecommand{\LLc}{\mathscr{L}}
\providecommand{\LLs}{\mathsf{L}}
\providecommand{\NNb}{\mathbb{N}}

\providecommand{\PPc}{\mathcal{P}}

\providecommand{\RRb}{\mathbb{R}}
\providecommand{\UUc}{\mathscr{U}}
\providecommand{\ZZb}{\mathbb{Z}}

\providecommand\llb{\llbracket}
\providecommand\rrb{\rrbracket}
\providecommand\maxrho{\rho}

\newcommand{\fixed}[2][1]{%
  \begingroup
  \spaceskip=#1\fontdimen2\font minus \fontdimen4\font
  \xspaceskip=0pt\relax
  #2%
  \endgroup
}

\newcommand{\myunderbar}[1]{\stackunder[1.2pt]{$#1$}{\rule{1ex}{.12ex}}}

\begin{document}

\title[Arithmetic of commutative semigroups]{Arithmetic of commutative semigroups \\ with a focus on  semigroups  of ideals and modules}

\author{Yushuang Fan}
\address{Mathematical College, China University of Geosciences | Haidian District, Beijing, China}
\email{fys@cugb.edu.cn}
\urladdr{http://yushuang-fan.weebly.com/}

\author{Alfred Geroldinger}
\address{Institute for Mathematics and Scientific Computing, University of Graz | Heinrichstr. 36, 8010 Graz, Austria}
\email{alfred.geroldinger@uni-graz.at}
\urladdr{http://imsc.uni-graz.at/geroldinger}

\author{Florian Kainrath}
\address{Institute for Mathematics and Scientific Computing, University of Graz | Heinrichstr. 36, 8010 Graz, Austria}
\email{florian.kainrath@uni-graz.at}

\author{Salvatore Tringali}
\address{Institute for Mathematics and Scientific Computing, University of Graz | Heinrichstr. 36, 8010 Graz, Austria}
\email{salvatore.tringali@uni-graz.at}
\urladdr{http://imsc.uni-graz.at/tringali/}

\thanks{This work was supported by the Austrian Science Fund FWF, Project Number M 1900-N39,  the National Natural Science Foundation of China (NSFC), Project No. 11401542, and  the China Scholarship Council (CSC)}

\keywords{Factorizations, sets of lengths, direct-sum decomposition, Krull monoids}

\subjclass[2010]{11P70, 13A05, 13F05, 16D70, 20M12, 20M13}

\begin{abstract}
Let $H$ be a commutative semigroup with unit element such that every non-unit can be written as a finite product of irreducible elements (atoms). For every $k \in \mathbb N$, let $\mathscr U_k (H)$ denote the set of all $\ell \in \mathbb N$ with the property that there are atoms $u_1, \ldots, u_k, v_1, \ldots,  v_{\ell}$ such that $u_1 \cdot \ldots \cdot u_k = v_1 \cdot \ldots \cdot v_{\ell}$ (thus, $\mathscr U_k (H)$ is the union of all sets of lengths containing $k$).

The Structure Theorem for Unions  states that, for all sufficiently large $k$, the sets $\mathscr U_k (H)$ are almost arithmetical progressions with the same difference and  global bound. We present a new approach to this result in the framework of arithmetic combinatorics, by deriving, for suitably defined families  of subsets of the non-negative integers, a characterization of when the Structure Theorem  holds.

This abstract approach allows us to verify, for the first time, the Structure Theorem for a variety of possibly non-cancellative  semigroups, including  semigroups of (not necessarily invertible) ideals and semigroups of modules. Furthermore, we provide the very first example of a semigroup (actually, a locally tame Krull monoid) that does not satisfy the Structure Theorem.
\end{abstract}

\maketitle
\thispagestyle{empty}

\medskip
\section{Introduction} \label{1}
\medskip

The focus of factorization theory has been so far  on the arithmetic of noetherian domains, Krull domains, semigroups of invertible ideals, and semigroups of modules, where the involved semigroups (of ring elements, ideals, or modules), both in the commutative and in the non-commutative setting, are cancellative. The best investigated means for describing the arithmetic of non-factorial semigroups are sets of lengths and invariants derived from them, such as sets of distances, elasticities, and unions of sets of lengths. To fix notation, we recall some definitions that lie at the center of our interest.

Let $H$ be a commutative semigroup and let $k$ be a positive integer. Then we use $\UUc_k (H)$ for the set of all $\ell \in \N$ with the property that there are atoms $u_1, \ldots, u_k, v_1, \ldots,  v_{\ell} \in H$ such that $u_1 \cdot \ldots \cdot u_k = v_1 \cdot \ldots \cdot v_{\ell}$. Moreover, $\rho_k (H) = \sup \UUc_k (H) \in \N_0 \cup \{\infty\}$ denotes the $k$-th elasticity of $H$.

Unions of sets of lengths were first studied by Chapman and Smith in the setting of Dedekind domains $R$ with finite class group  (\cite{Ch-Sm93b, Ch-Sm98a}).   Their focus was on the invariants  $|\UUc_k (R)|$ and $\rho_k (R)$.
The Structure Theorem for Unions states that, for all sufficiently large $k$, the sets $\mathscr{U}_k (H)$ are almost arithmetical progressions with the same difference and  global bound. It was first proved for commutative cancellative semigroups with finite set of distances satisfying a natural growth condition for the $\rho_k (H)$ (\cite[Theorem 4.2]{Ga-Ge09b}), and later generalized to possibly non-commutative cancellative semigroups (\cite[Theorem 2.6]{Ge16c}). These  abstract results served as crucial tools for studying unions of sets of lengths in settings ranging from numerical monoids to Mori domains, including semigroups of modules and maximal orders in central simple algebras  (see, for example, \cite{Ba-Ge14b, B-C-H-P08, Sm13a, Ge-Ka-Re15a}). In special cases the unions $\mathscr{U}_k (H)$ are not only almost arithmetical progressions (as predicted by the Structure Theorem), but even arithmetical progressions. This holds true, among others, for Krull monoids with the property that
every class contains a prime divisor (\cite[Theorem 3.1.3]{Ge09a}). In this case the study of the elasticities $\rho_k (H)$ has received much attention in recent literature (\cite{Ge-Gr-Yu15, Fa-Zh16a}).

The goal of the present paper is twofold. First, we want to gain a better understanding of when the Structure Theorem for Unions holds true, and second, we want to do this in a setting that allows us to handle certain semigroups of ideals and semigroups of modules which are not necessarily cancellative.

In Section \ref{2}, we study some suitably defined families of sets of non-negative integers (systems of sets of lengths of semigroups are such families), and for these families we derive a characterization of when the Structure Theorem for Unions holds true (Theorem \ref{th:structure_theorem}). In Section \ref{3}, we first consider abstract monoids, defined as commutative associative semigroups with a unit element for which an equation of the form $au=a$ implies that $u$ is invertible. Our main results in this setting are Theorems \ref{3.4} and  \ref{3.6}. In Subsections \ref{subsec:3.2} and \ref{subsec:3.3}, we apply the above results   to semigroups of ideals in one-dimensional noetherian domains (Corollary \ref{3.8}) and to semigroups of modules of finite representation type (Corollary \ref{3.9}). In Section \ref{4}, we provide the very first example of a semigroup whose system of sets of lengths does not satisfy the Structure Theorem for Unions. Our example is, actually, a cancellative locally tame Krull monoid with finite set of distances (Theorem \ref{4.2}).

\medskip
\section{Directed families of subsets of $\N_0$} \label{2}
\medskip

For $a,b \in \RRb \cup \{\pm \infty\}$ we let  $\llb a, b \rrb = \{x \in \Z \colon a \le x \le b \}$ be the discrete interval between $a$ and $b$.  Given a set $X$, we write $\PPc  (X)$ for the power set of $X$. We denote by $\N$ the set of positive integers and take $\N_0 = \N \cup \{0\}$. By convention, we assume $\inf \emptyset = \infty$ and $\sup \emptyset = 0$.

Let $G$ be an additive abelian group, $n \in \N$,  and $A, B \subseteq G$ subsets. Then $A+B = \{a+b \colon a \in A, b \in B\}$ is the sumset of $A$ and $B$, $nA = A+ \ldots + A$ the $n$-fold sumset of $A$, and $n \cdot A = \{na \colon a \in A\}$  the dilation of $A$ by $n$.

Let $L \subseteq \N_0$ be a subset. We denote by  $L^+ = L \cap \N$ the set of positive elements of $L$. The  \textit{elasticity} $\rho (L)$ of $L$ is then defined as $\rho (L)=\sup L^+/\min L^+$ if $L^+ \ne \emptyset$ and $\rho (L)=1$ if $L^+ = \emptyset$.  A positive integer $d \in \N$ is called a \textit{distance} of $L$ if there are $k, \ell \in L$ with $\ell - k = d$ and the interval $\llb k,\ell \rrb$ contains no further elements of $L$. We let $\Delta (L)$ denote the set of  distances of $L$. Thus $L$ is an arithmetical progression (AP for short) if $L \ne \emptyset$ and $|\Delta (L)|\le 1$.
Given $d \in \NNp$ and $M \in \NNb$, we call $L$
 an \textit{almost arithmetical progression} (AAP for short) with difference $d$ and bound $M$ if
\begin{itemize}
\item $L \subseteq y + d \cdot \ZZb \ $ for some $y \in \ZZb$, and
\item $L \cap \llb \inf L + M, \sup L - M \rrb$ is an AP with difference $d$.
\end{itemize}
Clearly, every AAP with difference $d$ and bound $M=0$ is an AP with difference $d$, and conversely. By definition, AAPs are non-empty and they may be either finite or infinite.

\smallskip
Let $\LLc$ be a family of  subsets of $\NNb$. For each $k \in \N_0$, we define
\[
\LLc_k = \{L \in \LLc: k \in L\} \quad \text{ and } \quad  \UUc_k(\LLc) = \bigcup_{L \in \LLc_k} L,
\]
and we set $\lambda_k(\LLc) = \inf \UUc_k(\LLc)$ and $\maxrho_k(\LLc) = \sup \UUc_k(\LLc)$. We call
\begin{itemize}
\item $\rho(\LLc) = \sup\{\rho(L): L \in \LLc\}$ the \textit{elasticity} of  $\LLc$,
\item $\Delta(\LLc) = \bigcup_{L \in \LLc} \Delta(L)$ the \textit{set of distances} of $\LLc$, and 
\item the elements of $\Delta(\LLc)$ the \textit{distances} of $\LLc$.
\end{itemize}
We say that the family $\LLc$
\begin{enumerate}[label={\rm (\roman{*})}]
\item\label{item:ST(i)} satisfies the \textit{Structure Theorem for Unions} if there are $d \in \NNp$ and $M \in \NNb$ such that $\UUc_k (\LLc)$ is an AAP with difference $d$ and bound $M$ for all sufficiently large $k \in \NNp$;
\item  is    \textit{directed} if $1 \in L$ for some $L \in \LLc$ and, for all $L_1, L_2 \in \LLc$, there is $L^\prime \in \LLc$ with $L_1 + L_2 \subseteq L^\prime$.
\end{enumerate}
In this section we study directed families of subsets of $\N_0$ and characterize when they satisfy the Structure Theorem for Unions (Theorem \ref{th:structure_theorem}).
Such families  play a prominent role in factorization theory, and here is the main example we have in mind.
\begin{example}
\label{exa:system_of_sets_of_lengths}
Let $\HHb$ be a multiplicative semigroup, $H^{\times}$ the group of invertible elements of $H$, and $A \subseteq H \setminus H^{\times}$ a non-empty subset such that every  $a \in H \setminus H^{\times}$ can be written as a finite product of elements from $A$. If $a \in H \setminus H^{\times}$, then
\[
\mathsf L (a) = \{ k \in \N \colon a = u_1 \cdot \ldots \cdot u_k \ \text{with} \ u_1, \ldots, u_k \in A \} \subseteq \N
\]
is the \textit{set of lengths} of $a$ (relative to $A$). For convenience we set $\mathsf L (a) = \{0\}$ for all $a \in H^{\times}$. Then $\LLc (H) = \{\mathsf L (a) \colon a \in H \}$ is the \textit{system of sets of lengths} of $H$ (relative to $A$). If $a \in A$, then $1 \in \mathsf L (a)$, and if $a_1, a_2 \in H$, then $\mathsf L (a_1)+\mathsf L (a_2) \subseteq \mathsf L (a_1a_2)$. Thus $\LLc (H)$ is a  directed family of subsets of $\N_0$.

The most natural choice of a generating set $A$ is the set of atoms of the semigroup and we study this situation in detail in Section \ref{3}. Also other generating sets have been, however,  studied in the literature  (e.g., \cite{C-D-H-K10}), and the main results of this section apply to them too. Clearly, every family $\LLc$ of subsets of $\N_0$, which is closed under set addition and  contains some $L$ with $1 \in L$, is directed (see Subsection \ref{subsec:3.4}). Yet, systems of sets of lengths are not, in general,  closed under set addition (\cite{Ge-Sc16b}).
\end{example}

With the above in mind, we present the main result of this section.

\begin{theorem}
\label{th:structure_theorem}
Let $\LLc$ be a  directed family of subsets of $\N_0$ such that $\Delta(\LLc)$ is finite non-empty and set $\delta = \min \Delta(\LLc)$.
\begin{enumerate}[label={\rm (\arabic{*})}]
\item\label{it:main(1)} Let $l \in \NNb$ such that $\{l, l + \delta\} \subseteq L$ for some $L \in \LLc$. Then $q = \frac{1}{\delta} \max \Delta(\LLc)$  is a non-negative integer and the following statements are equivalent{\rm \,:}
\begin{enumerate}[label={\rm (\alph{*})}, leftmargin=2.4\parindent]
\item\label{it:th:structure_theorem(i)} $\LLc$ satisfies the  Structure Theorem for Unions.
\item\label{it:th:structure_theorem(ii)} There exists $M \in \NNb$ such that $\UUc_k (\LLc) \cap \llb \maxrho_{k - lq}(\LLc) + lq, \maxrho_k (\LLc) - M \rrb$ is either empty or an AP with difference $\delta$ for all sufficiently large $k$.
\end{enumerate}
\smallskip
\item\label{it:main(2)} Assume that $\rho (\LLc) < \infty$ and there exists $L^\ast \in \LLc$ such that $\rho(L^\ast) = \rho (\LLc)$ and $L^\ast$ is an AP with difference $\delta$. Then $\UUc_k (\LLc)$ is an AP with dif\-fer\-ence $\delta$ for all sufficiently large $k$.
\end{enumerate}
\end{theorem}
Some comments about Condition \ref{it:main(1)}\ref{it:th:structure_theorem(ii)} of Theorem \ref{th:structure_theorem} are in order. Suppose first that $\rho_{k_0}(\LLc)=\infty$ for some $k_0 \in \N$. Then it is found that $\rho_k(\LLc) = \infty$ for all $k \ge k_0$ (see Lemma \ref{prop:basic(1)}\ref{it:prop:basic(1)(iii)}), and hence Condition \ref{it:main(1)}\ref{it:th:structure_theorem(ii)} in the theorem holds, because  $\llb \maxrho_{k - lq}(\LLc) + lq, \maxrho_k (\LLc) - M \rrb$ is empty.
Assume, on the other hand, that $\rho_k (\LLc) < \infty$ and $\rho_{k+1}(\LLc) - \rho_k (\LLc) \le K$ for all $k \in \N$, where $K \in \NNp$ does not depend on $k$. Then, in the notation of Theorem \ref{th:structure_theorem}, we have
$$
\maxrho_k (\LLc)- \maxrho_{k - lq} (\LLc)= \sum_{i=k-lq}^{k-1} (\maxrho_{i+1}(\LLc) - \maxrho_i (\LLc)) \le lqK,
$$
with the result that Condition \ref{it:main(1)}\ref{it:th:structure_theorem(ii)} is still verified (now with $M = lqK+1$).

\begin{corollary}
\label{cor:generalization}
Let $\LLc$ be a  directed family of subsets of $\N_0$ such that $\Delta(\LLc)$ is finite non-empty and the  Structure Theorem for Unions holds.
\begin{enumerate} [label={\rm (\arabic{*})}]
\item\label{it:cor:some_properties_of_ASTU_families(i)} If $\maxrho_k (\LLc) < \infty$ for all $k \in \NNp$, then there is an $M \in \N_0$ such that $\UUc_k (\LLc)$ is an AAP with difference $\min \Delta ( \LLc)$ and bound $M$ for every $k \in \N$.
\item\label{it:cor:some_properties_of_ASTU_families(ii)} We have that
\begin{equation}
\label{equ:limit_of_size}
\lim_{k \to \infty} \frac{|\UUc_k(\LLc)|}{k} = \sup_{k \ge 1} \frac{|\UUc_k(\LLc)|-1}{k} =
      \frac{1}{\min \Delta(\LLc)} \fixed[-0.15]{\text{ }} \left(\rho (\LLc) - \frac{1}{\rho (\LLc)}\right).
\end{equation}
\end{enumerate}
\end{corollary}
As mentioned in the Introduction, first results enforcing the validity of the Structure Theorem for Unions (together with the statements of Corollary \ref{cor:generalization}) were established in \cite[Theorem 4.2]{Ga-Ge09b} and \cite[Theorem 2.6]{Ge16c}. The formula (\ref{equ:limit_of_size}) was first derived in \cite{Ch-Sm93b}, for Dedekind domains with finite class group where each class contains a prime ideal. Note that, in this case, $\min \Delta (R)=1$ and $\rho (R) = \frac{1}{2}\mathsf D (G)$; hence the formula in \cite{Ch-Sm93b}  coincides with (\ref{equ:limit_of_size}). A result along the lines of point \ref{it:main(2)} of Theorem \ref{th:structure_theorem} was first established in \cite{Fr-Ge08}.

In order to prove Theorem \ref{th:structure_theorem}, we proceed in a series of lemmas.
For the rest of this section, $\LLc$ will be a directed family of subsets of $\N_0$. To ease notation,
we set
$$
\rho=\rho (\LLc)
\quad\text{and}\quad
\Delta = \Delta (\LLc),
$$
and for every $k \in \mathbb N_0$ we take
\[
\UUc_k = \UUc_k (\LLc), \quad  \rho_k = \rho_k (\LLc), \quad \text{and}\quad \lambda_k = \lambda_k (\LLc).
\]
If $n \in \N$ and $L_1, \ldots, L_n \in \LLc$, then for all $q_1, \ldots, q_n \in \N$  there is $L \in \LLc$ for which $q_1 L_1 + \ldots + q_n L_n \subseteq L$. In particular,  for every $k \in \NNp$ there exists $L \in \LLc$ with $k \in L$, whence $\LLc_k$ and $\UUc_k$ are non-empty.

\begin{lemma}
\label{prop:basic(1)}
Let  $h, k \in \N_0$. The following conditions hold:
\begin{enumerate}[label={\rm (\arabic{*})}]
\item\label{it:prop:basic(1)(i)} $h \in \UUc_k$ if and only if $k \in \UUc_h$.
\item\label{it:prop:basic(1)(ii)}
$\UUc_h + \UUc_k \subseteq \UUc_{h+k}
$.
\item\label{it:prop:basic(1)(iii)}
If $\mathscr{U}_h$ and $\mathscr{U}_k$ are non-empty, then $\lambda_{h+k} \le \lambda_h + \lambda_k \le h + k \le \maxrho_h + \maxrho_k \le \maxrho_{h+k}$.
\end{enumerate}
\end{lemma}
\begin{proof}
\ref{it:prop:basic(1)(i)} Let $h \in \UUc_k$. There then exists $L \in \LLc_k$ such that $h \in L$, and hence $L \in \LLc_h$. It follows that $k \in \UUc_h$, since $k \in L \subseteq \UUc_h$, and this is enough to conclude (by symmetry).

\ref{it:prop:basic(1)(ii)} The claim is obvious if $\UUc_h = \mathscr{U}_0 = \emptyset$ or $\UUc_k = \mathscr{U}_0 = \emptyset$. Otherwise, pick $x \in \UUc_h$ and $y \in \UUc_k$. Then $x \in L_h$ and $y \in L_k$ for some $L_h \in \LLc_h$ and $L_k \in \LLc_k$, and since $\LLc$ is directed, there exist $L' \in \LLc$ such that $x+y \in L_h + L_k \subseteq L'$. Then $L' \in \LLc_{h+k}$, and hence $x+y \in \UUc_{h+k}$, because $h+k \in L_h + L_k$. This implies that $\UUc_h + \UUc_k \subseteq \UUc_{h+k}$.

\ref{it:prop:basic(1)(iii)} It is sufficient to consider that, for all non-empty sets $X, Y \subseteq \RRb$, we have $\inf (X + Y) = \inf X + \inf Y$ and $\sup(X + Y) = \sup X + \sup Y$, and in addition $\inf Y \leq \inf X \le \sup X \le \sup Y$ if $X \subseteq Y$.
\end{proof}
\begin{lemma}
\label{lem:basic(2)}
Let $k \in \N$.
\begin{enumerate} [label={\rm (\arabic{*})}]
\item\label{it:lemma:basic(i)} $\maxrho_k = \sup\{\sup L: L \in \LLc,\fixed[0.4]{\text{ }} \inf L \le k\} = \sup\{\sup L: L \in \LLc_k\}$.
\item\label{it:lemma:basic(ii)} If $\maxrho_k < \infty$ and $\maxrho_k = \sup L$ for some $L \in \LLc$ with $\inf L \le k$, then $\inf L = k$ and $L \in \LLc_k$.
\item\label{it:lemma:basic(iii)} $\maxrho_k \ge \sup\{\sup L: L \in \LLc,\fixed[0.3]{\text{ }} \inf L = k\}$, and equality holds if $\maxrho_k < \infty$.
\end{enumerate}
\end{lemma}
\begin{proof}
For ease of exposition, set $\maxrho_k^{\fixed[0.2]{\text{ }}\prime} = \sup\{\sup L: L \in \LLc,\fixed[0.3]{\text{ }} \inf L \le k\}$, $\maxrho_k^{\fixed[0.2]{\text{ }}\pprime} = \sup\{\sup L: L \in \LLc_k\}$, and $\bar{\maxrho}_k = \sup\{\sup L: L \in \LLc,\fixed[0.3]{\text{ }} \inf L = k\}$. Then, pick $L_\ast \in \LLc$ such that $1 \in L_\ast$.

\ref{it:lemma:basic(i)} We have to show that $\maxrho_k = \maxrho_k^{\fixed[0.2]{\text{ }}\prime} = \maxrho_k^{\fixed[0.2]{\text{ }}\pprime}$. To begin, suppose for a contradiction that $\maxrho_k^{\fixed[0.2]{\text{ }}\pprime} < \maxrho_k$. There then exists $h \in \UUc_k$ such that $\maxrho_k^{\fixed[0.2]{\text{ }}\pprime} < h \le \maxrho_k$. But $\UUc_k = \bigcup_{L \in \LLc_k} L$ yields that $h \in L$ for some $L \in \LLc_k$, with the result that $\sup L \le \maxrho_k^{\fixed[0.2]{\text{ }}\pprime} < h \le \sup L$, which is impossible.

It follows that $\maxrho_k \le \maxrho_k^{\fixed[0.2]{\text{ }}\pprime}$, and it is clear from the definitions that $\maxrho_k^{\fixed[0.2]{\text{ }}\pprime} \le \maxrho_k^{\fixed[0.2]{\text{ }}\prime}$, so we are left to show that $\maxrho_k^{\fixed[0.2]{\text{ }}\prime} \le \maxrho_k$. For this, assume to the contrary that $\maxrho_k < \maxrho_k^{\fixed[0.2]{\text{ }}\prime}$. Then there is a set $L \in \LLc$ for which
\begin{equation}
\label{equ:basic(2)}
\inf L \le k
\quad\text{and}\quad
\maxrho_k < \sup L \le \maxrho_k^{\fixed[0.2]{\text{ }}\prime}.
\end{equation}
In particular, there exists $l \in L$ with $l \le k$, and $\LLc$ being directed implies that $k = l + (k-l) \in L + (k-l)L_\ast \subseteq \bar L$ for some $\bar L \in \LLc$. Thus $\bar{L} \in \LLc_k$, viz. $\bar{L} \subseteq \UUc_k$, and we have $\sup L \le \sup \bar{L} \le \maxrho_k$, which contradicts \eqref{equ:basic(2)} and concludes the proof of point \ref{it:lemma:basic(i)} of the lemma.

\ref{it:lemma:basic(ii)} Let $\maxrho_k = \sup L < \infty$ for some $L \in \LLc$ with $\inf L \le k$. Accordingly, let $l \in L$ such that $l \le k$. As before, we have (by the directedness of $\LLc$) that $k \in L + (k-l)L_\ast \subseteq \bar L$ for some $\bar L \in \LLc$. This implies $\maxrho_k \ge \sup \bar{L} \ge \sup L + k - l \ge \maxrho_k$, which is possible only if $l = k$. But $l$ is an arbitrary integer in $L$ that is $\le k$, so we can conclude that $\inf L = k$, and hence $L \in \LLc_k$.

\ref{it:lemma:basic(iii)} It is clear from point \ref{it:lemma:basic(i)} that $\bar{\maxrho}_k \le \maxrho_k^{\fixed[0.2]{\text{ }}\prime} = \maxrho_k$, and the rest is straightforward by point \ref{it:lemma:basic(ii)}.
\end{proof}

The next propositions provide an interesting connection among the parameters we have so far introduced. In the proof of the first of them, we will make use of the following result, often referred to as  Fekete's Lemma (see \cite[Pt. I, Chap. 3, Problem No. 98, pp. 23 and 198]{Po-Sz04}).
\begin{lemma}
Let $(x_k)_{k \ge 1}$ be a sequence with values in $\RRb \cup \{\infty\}$ such that $x_h + x_k \le x_{h+k}$ for all $h, k \in \NNp$. Then the limit of $\frac{1}{k}x_k$ as $k \to \infty$ exists and is equal to $\sup_{k \ge 1} \frac{1}{k}x_k$.
\end{lemma}
\begin{proposition}
\label{prop:limits}
We have
\begin{equation*}
\label{equ:rho=sup}
\rho = \sup_{k \ge 1} \frac{\maxrho_k}{k} = \lim_{k \to \infty} \frac{\maxrho_k}{k}
\quad\text{and}\quad
\frac{1}{\rho} = \inf_{k \ge 1} \frac{\lambda_k}{k} = \lim_{k \to \infty} \frac{\lambda_k}{k}.
\end{equation*}
In particular, if $\Delta \ne \emptyset$, then $\rho > 1$ and $\min \{k - \lambda_k, \maxrho_k - k \} \to \infty$ as $k \to \infty$.
\end{proposition}
\begin{proof}
The second part of the statement is trivial, and we can just focus on the limits. To this end, we have by Lemma \ref{prop:basic(1)} that $\lambda_{h+k} \le \lambda_h + \lambda_k \le \maxrho_h + \maxrho_k \le \maxrho_{h+k}$ for all $h, k \in \NNp$.

So, applying Fekete's Lemma to the sequences $(\maxrho_k)_{k \ge 1}$ and $(-\lambda_k)_{k \ge 1}$, we find that the limits of $\frac{1}{k} \maxrho_k$ and $\frac{1}{k} \lambda_k$ as $k \to \infty$ exist and are equal, respectively, to $\varrho$ and $\lambda$, where for ease of notation we set
$$\varrho := \sup_{k \ge 1} \frac{1}{k} \maxrho_k
\quad\text{and}\quad
\lambda := \inf_{k \ge 1} \frac{1}{k} \lambda_k.
$$
We are left to show that $\rho = \varrho$ and $1/\rho = \lambda$.
\vskip 0.1cm
\noindent{}\textsc{Case 1:} $\rho = \infty$.

For each $\varepsilon \in \RRb^+$ there exists $L_\varepsilon \in \LLc$ with $L_\varepsilon^+ \ne \emptyset$ and $\rho(L_\varepsilon) > \frac{1}{\varepsilon}$. So, taking $k_\varepsilon = \min L_\varepsilon^+$ yields that $\varrho \ge \frac{1}{k_\varepsilon} \maxrho_{k_\varepsilon} \ge \rho(L_\varepsilon) > \frac{1}{\varepsilon}$,
whence $\varrho = \infty$.
It remains to prove that $\lambda = 0$. This is obvious if $\lambda_k = 0$ for some $k \in \NNp$.
Otherwise, we have from the above that $k_\varepsilon = \min L_\varepsilon$ (i.e., $L_\varepsilon = L_\varepsilon^+$) and there is $K_\varepsilon \in L_\varepsilon$ such that $K_\varepsilon/k_\varepsilon \ge \frac{1}{\varepsilon}$. Thus we find that
$$
0 \le \lambda \le \frac{\lambda_{K_\varepsilon}}{K_\varepsilon} \le \frac{1}{K_\varepsilon} \min L_\varepsilon \le \varepsilon,
$$
which is enough to conclude that $\lambda = 0$.
\vskip 0.1cm
\noindent{}\textsc{Case 2:} $\rho < \infty$.

If $\varrho > \rho$, then $\frac{1}{k} \maxrho_k > \rho$ for some $k \in \NNp$, and this in turn implies that there exists $\bar{L} \in \LLc_k$ with $\bar{L}^+ \ne \emptyset$ such that $\rho < \frac{1}{k} \sup \bar{L} \le \sup \bar{L}/\min \bar{L}^+ = \rho(\bar{L})$, which is, however, a contradiction, because $\rho(\bar{L}) \le \rho$.
If, on the other hand, $\varrho < \rho$, then there is $\myunderbar{L} \in \LLc$ with $\myunderbar{L}^+ \ne \emptyset$ and $\varrho < \rho(\myunderbar{L}) = \sup \myunderbar{L}/\min \myunderbar{L}^+$, and taking $k = \min \myunderbar{L}^+$ yields  $\varrho < \frac{1}{k} \maxrho_k \le \varrho$, which is still impossible.
So we see that $\rho = \varrho$, and hence $\maxrho_1, \maxrho_2, \ldots$ are all finite.
In particular, for each $k \in \NNp$ we can find a set $L_k \in \LLc_k$ such that $\maxrho_k = \max L_k$, and by Lemma \ref{lem:basic(2)}\ref{it:lemma:basic(ii)} we have $\inf L_k = \min L_k = k$. It follows that $k \in \UUc_{\maxrho_k}$, and hence $\lambda_{\maxrho_k} \le k$, for all $k \in \NNp$. Moreover,
\begin{equation}
\label{equ:bound-from-below-on-lambda}
\lambda \ge \inf_{k \ge 1} \frac{\lambda_k}{\sup L_k} = \inf_{k \ge 1} \frac{1}{\rho(L_k)} \ge \frac{1}{\rho}.
\end{equation}
With this in mind, pick an integer $k \ge 1 + \maxrho_1$ and let $l_k$ be the maximal $l \in \NNp$ such that $\maxrho_l < k$. Then $k = \maxrho_{l_k} + j_k$ for some $j_k \in \llb 1, \maxrho_{l_k+1} - \maxrho_{l_k} \rrb$, and since the function
$$
{[0,\infty[} \to \RRb, x \mapsto \frac{a+x}{b+x}
$$
is non-decreasing whenever $a, b \in \RRb^+$ and $a \le b$, we obtain from \eqref{equ:bound-from-below-on-lambda} and Lemma \ref{prop:basic(1)} that
\begin{equation*}
\label{equ:chain_of_inequalities}
\begin{split}
\frac{1}{\rho} \le \lambda &
    \le \frac{\lambda_k}{k} \le \frac{\lambda_{\maxrho_{l_k}} + \lambda_{j_k}}{\maxrho_{l_k} + j_k} \le \frac{l_k + j_k}{\maxrho_{l_k} + j_k} \le \frac{l_k + \maxrho_{l_k+1} - \maxrho_{l_k}}{\maxrho_{l_k} + \maxrho_{l_k+1} - \maxrho_{l_k}} = \frac{l_k}{\maxrho_{l_k+1}} + 1 - \frac{\maxrho_{l_k}}{\maxrho_{l_k+1}}.
\end{split}
\end{equation*}
But the rightmost side of this last equation tends to $1/\rho$ as $k \to \infty$, and hence $\lambda = 1/\rho$.
\end{proof}
\begin{proposition}
\label{prop:accepter_elasticity}
Suppose that $\LLc$ has elasticity $\rho  < \infty$. Then the following statements are equivalent{\rm \,:}
\begin{enumerate}[label={\rm (\alph{*})}]
\item \label{it:equivalents(i)} There is $L \in \LLc$ such that $L^+ \ne \emptyset$ and $\rho(L) = \rho$.
\item\label{it:equivalents(ii)} There exists $n \in \NNp$ such that $nk \rho = \maxrho_{nk}$ for all $k \in \NNp$.
\item\label{it:equivalents(iii)} There is some $k \in \NNp$ such that $k \rho = \maxrho_k$.
\end{enumerate}
Moreover, if at least one of  conditions \ref{it:equivalents(i)}-\ref{it:equivalents(iii)} is verified, there exists $K \in \NNb$ such that $\maxrho_{k+1} \le \maxrho_k + K$ for every $k \in \NNp$.
\end{proposition}
\begin{proof} Assume first that $\rho(L) = \rho$ for some $L \in \LLc$ with $L^+ \ne \emptyset$, and set $n = \min L^+$. Since $\LLc$ is directed, it follows that for every $k \in \NNp$ there exists $L_k \in \LLc$ such that $nk \in kL \subseteq L_k$, and thus
$$
\frac{\maxrho_{nk}}{nk} \ge \frac{\sup L_k}{nk} \ge \frac{k \sup L}{nk} = \rho(L) = \rho \ge \frac{\maxrho_{nk}}{nk},
$$
where the last inequality comes as a consequence of Proposition \ref{prop:limits}. This proves that \ref{it:equivalents(i)} $\Rightarrow$ \ref{it:equivalents(ii)}.

Suppose now that \ref{it:equivalents(iii)} holds, and let $k \in \NNp$ such that $k \rho = \maxrho_k$. Then $\maxrho_k = \max L$ for some $L \in \LLc_k$ with $L^+ \ne \emptyset$. So $k = \min L$ by Lemma \ref{lem:basic(2)}\ref{it:lemma:basic(ii)}, which gives $k \rho = \max L = k\fixed[0.15]{\text{ }} \rho(L)$, viz. $\rho(L) = \rho$.
Putting it all together, we thus see that \ref{it:equivalents(iii)} $\Rightarrow$ \ref{it:equivalents(i)}.

Since, on the other hand, it is obvious that \ref{it:equivalents(ii)} $\Rightarrow$ \ref{it:equivalents(iii)}, we are just left to show the second part of the statement.
To this end, let $n \in \NNp$ with the property that $nk \rho = \maxrho_{nk}$ for all $k \in \NNp$.

Given $k \in \NNp$, there are $q_k \in \NNb$ and $r_k \in \llb 1, n \rrb$ such that $k = nq_k + r_k$, so we get from the above and Lemma \ref{prop:basic(1)}\ref{it:prop:basic(1)(iii)} that
\begin{equation*}
\begin{split}
k \rho - n(\rho - 1)
    &
    \le (nq_k + r_k)\rho - r_k (\rho - 1)
    = nq_k \rho + r_k
    \le \maxrho_{nq_k} + \maxrho_{r_k} \le \maxrho_{nq_k + r_k} = \maxrho_k,
\end{split}
\end{equation*}
namely $k \rho \le \maxrho_k + n(\rho - 1)$. It follows that
$$
\maxrho_{k+1} \le (k+1) \rho \le \maxrho_k + n(\rho - 1) + \rho
$$
for every $k \in \NNp$, which completes the proof by taking $K = (n+1)\rho - n$.
\end{proof}
\begin{proposition}
\label{prop:delta_sets}
Let  $\Delta' \subseteq \Delta$ be a non-empty subset with $\gcd(\Delta') \le \min \Delta$. Then $\gcd(\Delta') = \min \Delta$. In particular, $\gcd \Delta = \min \Delta$.
\end{proposition}
\begin{proof}
Using that $\Delta'$ is non-empty, define $\delta = \gcd(\Delta')$. There exist $\eps_1, \ldots, \eps_n \in \{\pm 1\}$, $d_1, \ldots, d_n \in \Delta'$, and $m_1, \ldots, m_n \in \NNp$ such that
$\delta = \eps_1 m_1 d_1 + \ldots + \eps_n m_n d_n$.

In addition, for each $i = 1, \ldots, n$ we can find $x_i \in \NNb$ and $L_i \in \LLc$ such that $\{x_i, x_i + \eps_i d_i\} \subseteq L_i$. Since $\LLc$ is directed, this gives that
$
\{m_i x_i, m_i (x_i + \eps_i d_i)\} \subseteq
m_i L_i \subseteq L_i'$ for some $L_i' \in \LLc$. Moreover, there is $L \in \LLc$ such that $L_1' + \ldots + L_n' \subseteq L$. Put $l = m_1 x_1 + \ldots + m_n x_n$.

Then we have by the above that
$
l + \delta =
\sum_{i=1}^n m_i (x_i + \eps_i d_i)
$,
and we see that both $l$ and $l+\delta$ belong to $L$. Thus $\min \Delta \le \min \Delta(L) \le \delta = \gcd(\Delta')$, and this is enough to conclude that $\gcd(\Delta') = \min \Delta$, as on the other hand we have by hypothesis that $\gcd(\Delta') \le \min \Delta$.

The rest is clear, because $\Delta$ being non-empty implies that $\gcd \Delta \le \min \Delta$.
\end{proof}
\begin{lemma}
\label{lem:arbitrary_long_APs}
Let $d \in \Delta$.
\begin{enumerate} [label={\rm (\arabic{*})}]
\item\label{it:arbitrary_long_APs(i)} There exists $l \in \mathbb{N}_0$ such that, for every $q \in \NNp$, there is a set $L_q \in \LLc$ with $lq + d \cdot \llb 0, q  \rrb \subseteq L_q \subseteq \UUc_{l q} \cap \llb lq, \maxrho_{lq} \rrb$.
\item\label{it:arbitrary_long_APs(ii)} Let $q \in \NNp$. Then for every sufficiently large $k \in \NNp$ there exists $L_k \in \LLc$ such that $k + d \cdot \llb 0, q \rrb \subseteq L_k \subseteq \UUc_k \cap \llb k, \maxrho_k \rrb$.
\end{enumerate}
\end{lemma}
\begin{proof}
Let $L_0 \in \mathscr{L}$ and $l \in \mathbb{N}_0$ such that $\{l, l+d\} = L_0 \cap \llb l, l+d \rrb$.

\ref{it:arbitrary_long_APs(i)} Using that $\mathscr{L}$ is a directed family, we see that, for every $q \in \mathbb{N}$, there is a set $L \in \LLc$ such that
$lq + d \cdot \llb 0, q \rrb = q \fixed[0.25]{\text{ }} \{l, l + d\} \subseteq q\fixed[0.1]{\text{ }} L_0 \subseteq L \subseteq \UUc_{lq}$. The rest is obvious.

\ref{it:arbitrary_long_APs(ii)} Let $k$ be an integer $\ge lq+1$. Clearly, there exists $L^\prime \in \LLc$ such that $k-lq \in L^\prime$. On the other hand, it follows from the proof of point \ref{it:arbitrary_long_APs(i)} that there is $L^\pprime \in \LLc$ with $lq + d \cdot \llb 0, q  \rrb \subseteq L^\pprime$. Thus, we have that $k + d \cdot \llb 0, q \rrb \subseteq L^\prime + L^\pprime \subseteq L \subseteq \UUc_k$ for some $L \in \LLc$, and the rest is clear.
\end{proof}
\begin{lemma}
\label{lem_FrGe}
Suppose    $\Delta$ is finite non-empty, set $\delta = \min \Delta$, and let $k \in \N$.
\begin{enumerate} [label={\rm (\arabic{*})}]
\item\label{it:lemma_FrGe(i)} $\UUc_k \subseteq k + \delta \cdot \ZZb$ and $\Delta (\UUc_k) \subseteq \delta \cdot \NNp$.
\item\label{it:lemma_FrGe(ii)} $\sup \Delta(\UUc_k) \le \sup \Delta$.
\end{enumerate}
\end{lemma}
\begin{proof}
\ref{it:lemma_FrGe(i)} Proposition \ref{prop:delta_sets} implies that, for every $l \in \mathscr{U}_k$, $l-k$ is a multiple of $\delta$, and this yields that $\UUc_k \subseteq k + \delta \cdot \ZZb$. The rest is obvious.

\ref{it:lemma_FrGe(ii)} Assume to the contrary that $\sup \Delta(\UUc_k) > \sup \Delta$. Then $\UUc_k \ne \emptyset$ and there exist $l_1, l_2 \in \NNb$ with $l_2 - l_1 > \sup \Delta$ and $\{l_1, l_2\} = \UUc_k \cap \llb l_1, l_2 \rrb$. Accordingly, there also exist $L_1, L_2 \in \LLc_k$ such that $l_1 \in L_1$ and $l_2 \in L_2$. We distinguish two cases.
\vskip 0.1cm
\textsc{Case 1:} $\min L_2 = l_2$.

Then $l_1 < l_2 \le k$, and since $l_1, k \in L_1$, there exists $m _1\in L_1$ with $l_1 < m_1$ and $\llb l_1, m_1 \rrb \cap L_1 = \{l_1, m_1\}$. But $L_1 \subseteq \UUc_k$, so $l_2 \le m_1$, and hence $l_2 - l_1 \le m_1 - l_1 \in \Delta(L_1) \subseteq \Delta$, viz. $l_2 - l_1 \le \sup \Delta$, a contradiction.
\vskip 0.1cm
\textsc{Case 2:} $\min L_2 < l_2$.

Then there is $m_2 \in L_2$ with $m_2 < l_2$ and $\llb m_2, l_2 \rrb \cap L_2 = \{m_2, l_2\}$. But $L_2 \subseteq \UUc_k$, so we must have that $m_2 \le l_1$, and hence $l_2 - l_1 \le l_2 - m_2 \in \Delta(L_2) \subseteq \Delta$, which is still a contradiction.
\end{proof}
\begin{lemma}
\label{lem:reduction_to_the_upper_part}
Suppose   $\Delta$ is finite non-empty and let $d \in \NNp$. Then the following statements are equivalent{\rm \,:}
\begin{enumerate}[label={\rm (\alph{*})}]
\item\label{it:lem:reduction_to_the_upper_part(i)} There exists $M_{(a)} \in \NNb$ such that $\UUc_k$ is an AAP with difference $d$ and bound $M_{(a)}$ for all sufficiently large $k$.
\item\label{it:lem:reduction_to_the_upper_part(ii)} There is $M_{(b)} \in \NNb$ such that $\UUc_k \cap \bigl\llb k - M_{(b)}, \maxrho_k \bigr\rrb$ is an AAP with difference $d$ and bound $M_{(b)}$ for all sufficiently large $k$.
\item\label{it:lem:reduction_to_the_upper_part(iii)} There exists $M_{(c)} \in \NNb$ such that $\UUc_k \cap \bigl\llb k, \maxrho_k - M_{(c)} \bigr\rrb$ is an AP with difference $d$ for all sufficiently large $k$.
\end{enumerate}
In addition, if at least one of conditions \ref{it:lem:reduction_to_the_upper_part(i)}-\ref{it:lem:reduction_to_the_upper_part(iii)} is satisfied, then $d = \min \Delta$.
\end{lemma}
\begin{proof}
Set $\delta = \min \Delta$ for brevity's sake. Given $M \in \NNb$, we have by Proposition \ref{prop:limits} and Lemma \ref{lem:arbitrary_long_APs}\ref{it:arbitrary_long_APs(ii)} that there exists $k_M \in \N$ such that
\begin{equation}
\label{equ:some_conditions}
k \ge \lambda_k + \max \{\delta, M\}
\quad \text{and}\quad
k + \delta \cdot \llb 0, M+1 \rrb \subseteq \UUc_k \cap \llb k, \maxrho_k \rrb
\quad \text{for }k \ge k_M.
\end{equation}
It is thus clear that if at least one of conditions \ref{it:lem:reduction_to_the_upper_part(i)}-\ref{it:lem:reduction_to_the_upper_part(iii)} is verified, then $d = \delta$, and this in turn makes it plain that \ref{it:lem:reduction_to_the_upper_part(ii)} $\Leftrightarrow$ \ref{it:lem:reduction_to_the_upper_part(iii)}. Therefore it is enough to prove that \ref{it:lem:reduction_to_the_upper_part(i)} $\Leftrightarrow$ \ref{it:lem:reduction_to_the_upper_part(iii)} with $d = \delta$.
\vskip 0.1cm
\ref{it:lem:reduction_to_the_upper_part(i)} $\Rightarrow$ \ref{it:lem:reduction_to_the_upper_part(iii)}. Let $M_{(a)} \in \NNb$ and $k_\ast \in \N$ such that, for every $k \ge k_\ast$, $\UUc_k$ is an AAP with difference $\delta$ and bound $M_{(a)}$. Also, let $k_a \in \N$ such that \eqref{equ:some_conditions} holds with $M = M_{(a)}$ and $k_M = k_a$.

Then, for each $k \ge \max \{k_\ast, k_a\}$ we have that $\UUc_k^\ast = \UUc_k \cap \bigl\llb \lambda_k + M_{(a)}, \maxrho_k - M_{(a)} \bigr\rrb$ is an AP with difference $\delta$ and, in addition, $k \in \UUc_k^\ast$. It is then evident that, for all sufficiently large $k$, also $\UUc_k \cap \bigl\llb k, \maxrho_k - M_{(a)} \bigr\rrb$ is an AP with difference $\delta$.
\vskip 0.1cm
\ref{it:lem:reduction_to_the_upper_part(iii)} $\Rightarrow$ \ref{it:lem:reduction_to_the_upper_part(i)}. Let $M_{(c)}  \in \NNb$ and $k_c \in \N$ be such that $\UUc_k \cap \bigl\llb k, \maxrho_k - M_{(c)} \bigr\rrb$ is an AP with difference $\delta$ for every $k \ge k_c$. Then set $M_{(a)} = \max \bigl\{k_c, M_{(c)}\bigr\}$ and let $k_a \in \mathbb N$ such that \eqref{equ:some_conditions} is satisfied with $M = M_{(a)}$ and $k_M = k_a$.

Pick $k \ge \max \{k_a, k_c \}$ and
$m \in \bigl\llb \lambda_k + M_{(a)}, k \bigr\rrb$ such that $k-m$ is a multiple of $\delta$.
Clearly $k \le \maxrho_{\lambda_k}$ and $\lambda_k \le m - M_{(a)}$. So it follows from Lemma \ref{prop:basic(1)}\ref{it:prop:basic(1)(iii)} that $k \le \maxrho_{\lambda_k} \le \maxrho_{m-M_{(a)}}$, and hence $k \le \maxrho_m - M_{(a)}$, because
$$
k + M_{(a)} \le \maxrho_{m - M_{(a)}} + M_{(a)} \le \maxrho_{m - M_{(a)}} + \maxrho_{M_{(a)}}
 \le \maxrho_m.
$$
On the other hand, $k \in m + \delta \cdot \NNb$ and
$\UUc_m \cap \bigl\llb m, \maxrho_m - M_{(a)} \bigr\rrb$ is an AP with difference $\delta$, since $m \ge \lambda_k + M_{(a)} \ge k_c$ and $M_{(a)} \ge M_{(c)}$, with the result that $\bigl\llb k, \maxrho_k - M_{(a)} \bigr\rrb \subseteq \bigl\llb k, \maxrho_k - M_{(c)} \bigr\rrb$. Putting it all together, we thus see that $k \in \UUc_m$, and hence $m \in \UUc_k$ by Lemma \ref{prop:basic(1)}\ref{it:prop:basic(1)(i)}.

So, by the arbitrariness of $m \in \bigl\llb \lambda_k + M_{(a)}, k \bigr\rrb$ and the above, we find that
$$
\bigl\llb \lambda_k + M_{(a)}, \maxrho_k - M_{(a)} \bigr\rrb \cap (k + \delta \cdot \ZZb) = \UUc_k \cap \bigl\llb \lambda_k + M_{(a)}, \maxrho_k - M_{(a)} \bigr\rrb,
$$
and this implies that $\UUc_k$ is an AAP with difference $\delta$ and bound $M_{(a)}$.
\end{proof}
\begin{lemma}
\label{lem:AP_reduction_to_the_upper_part}
Suppose  $\LLc$ has elasticity $\rho < \infty$ and  $\Delta$ is finite non-empty.
\begin{enumerate}[label={\rm (\arabic{*})}]
\item\label{item:AP_reduction(i)} We have $\lambda_k \ge \rho^{-1} k > 0$ for every $k \ge 1 + \rho \max \Delta$.

\item\label{item:AP_reduction(ii)} For $d \in \NNp$,  the following statements are equivalent:
\begin{enumerate}[label={\rm (\alph{*})}, leftmargin=2.4\parindent]
\item\label{it:lem:AP_reduction_to_the_upper_part(i)} $\UUc_k$ is an AP with difference $d$ for all large $k$.
\item\label{it:lem:AP_reduction_to_the_upper_part(ii)} $\UUc_k \cap \bigl\llb k, \maxrho_k \bigr\rrb$ is an AP with difference $d$ for all large $k$.
\end{enumerate}
In addition, if either of conditions \ref{it:lem:AP_reduction_to_the_upper_part(i)} and \ref{it:lem:AP_reduction_to_the_upper_part(ii)} holds, then $d = \min \Delta$.
\end{enumerate}
\end{lemma}
\begin{proof}
\ref{item:AP_reduction(i)} Let $k \ge 1 + \rho \max \Delta$. There then exists $L \in \LLc$ such that $\lambda_k, k \in L$, and we have $k \le \sup L \le \rho \min L^+$, that is $\min L^+ \ge \rho^{-1} k > \max \Delta$. It follows that $L^+ = L$, and hence $\lambda_k = \min L^+ \ge \rho^{-1} k$, since $0 \in L$ would imply that $\max \Delta \ge \min L^+ > \max \Delta$, which is impossible.

\smallskip
\ref{item:AP_reduction(ii)} We set  $\delta = \min \Delta$. The last part of the claim (``In addition, if etc.'') is clear from Lemma \ref{lem:reduction_to_the_upper_part}, and on the other hand it is obvious that \ref{it:lem:AP_reduction_to_the_upper_part(i)} $\Rightarrow$ \ref{it:lem:AP_reduction_to_the_upper_part(ii)}. So we assume for the sequel that $d = \delta$ and we show that \ref{it:lem:AP_reduction_to_the_upper_part(ii)} $\Rightarrow$ \ref{it:lem:AP_reduction_to_the_upper_part(i)}.

To start with, let $k_0^\prime \in \NNp$ such that $\lambda_k \ge \rho^{-1} k$ for $k \ge k_0^\prime$, which is possible by \ref{item:AP_reduction(i)}. Then let $k_0^\pprime \in \NNb$ such that $\UUc_k \cap \llb k, \maxrho_k \rrb$ is an AP for $k \ge k_0^\pprime$, and pick $k \ge \max \{k_0^\prime, \rho k_0^\pprime \}$. Since $k \in \UUc_k$, it is enough to prove that $\UUc_k \cap \llb \lambda_k, k \rrb$ is an AP with
difference $\delta$.

To this end, let $m \in \llb \lambda_k, k \rrb$ such that $k - m$ is a multiple of $\delta$. Then $k_0^\pprime \le \rho^{-1} k \le \lambda_k$, and hence $\UUc_m \cap \llb m, \rho_m \rrb$ is an AP with difference $\delta$. But $m \le k \le \maxrho_{\lambda_k} \le \maxrho_m$ and $k \in m + \delta \cdot \NNb$, so $k \in \UUc_m$, and this implies that $m \in \UUc_k$ by Lemma \ref{prop:basic(1)}\ref{it:prop:basic(1)(i)}.
\end{proof}

\begin{proof}[Proof of Theorem \ref{th:structure_theorem}]~

\smallskip
\ref{it:main(1)} Let $k$ be an arbitrary integer $\ge lq + 1$.
\vskip 0.1cm
\ref{it:th:structure_theorem(i)} $\Rightarrow$ \ref{it:th:structure_theorem(ii)}. By hypothesis, there are $M \in \NNb$ and $k_0 \in \NNp$ such that $\UUc_k \cap \llb \lambda_k + M, \maxrho_k - M \rrb$ is an AP with difference $\delta$ for all $k \ge k_0$. On the other hand, we get from Proposition \ref{prop:limits} that there also exists $k_\ast \in \NNp$ such that $k \ge \lambda_k + M$ and $\maxrho_{k-lq} \ge k$ for $k \ge k_\ast$. So, putting it all together, we see that $\UUc_k \cap \llb \maxrho_{k-lq} + lq, \maxrho_k - M \rrb$ is either empty or an AP with difference $\delta$ for all sufficiently large $k \in \NNp$.
\vskip 0.1cm
\ref{it:th:structure_theorem(ii)} $\Rightarrow$ \ref{it:th:structure_theorem(i)}.
We know from Proposition \ref{prop:delta_sets} that $\delta = \gcd \Delta$ and $\Delta \subseteq \delta \cdot \NNp$. In particular, this gives that $\UUc_k \subseteq \lambda_k + \delta \cdot \NNb$ and $q$ is a positive integer, since $q = \frac{1}{\delta} \max \Delta$ and $\max \Delta$ is a multiple of $\delta$. Moreover, $\LLc$ being a directed family implies that
\begin{equation}
\label{equ:inclusion(1)}
\UUc^* := lq + \delta \cdot \fixed[-0.1]{\text{ }} \llb 0, q \rrb
\subseteq \UUc_{lq}.
\end{equation}
By Lemma \ref{lem:reduction_to_the_upper_part}, it is enough to show that there exists $k^\ast \in \mathbb N$ such that the set $\UUc_k \cap \llb k, \maxrho_k - M \rrb$ is an AP with difference $\delta$ for $k \ge k^\ast$.
To this end, let
$$
\UUc^\prime = \UUc_{lq} \cap \llb \lambda_{lq}, lq-1 \rrb
\quad\text{and}\quad
\UUc^\pprime = \UUc_{lq} \cap \llb (l+\delta)q+1,\maxrho_{lq} \rrb,
$$
and note that $\UUc^\prime$, $\UUc^\ast$, and $\UUc^\pprime$ are pairwise disjoint and
$\UUc_{lq} = \UUc^\prime \cup \UUc^* \cup \UUc^\pprime$.
It is evident that $k \in \UUc^* + \UUc_{k-lq}$, and Lemma \ref{prop:basic(1)}\ref{it:prop:basic(1)(ii)} gives
\begin{equation}
\label{equ:inclusion(2)}
\UUc_{lq} + \UUc_{k-lq}
    = \fixed[-0.2]{\text{ }} (\UUc^\prime + \UUc_{k-lq} ) \cup (\UUc^\ast + \UUc_{k-lq}) \cup (\UUc^\pprime + \UUc_{k-lq} ) \subseteq \UUc_k.
\end{equation}
On the other hand, we have by Lemma \ref{lem_FrGe} that $\max
\Delta(\UUc_{k-lq}) \le \max \Delta$
and $\Delta(\UUc_{k-lq}) \subseteq \delta \cdot \NNp$.

It follows that $\UUc^* + \UUc_{k-lq}$ is an AP with difference $\delta$, and it is straightforward that $\sup(\UUc^* + \UUc_{k-lq}) = lq + \delta q + \maxrho_{k-lq}$. Thus, if we let $k^\prime \in \NNb$ such that $\rho_k \ge M + k$ for $k \ge k^\prime$ (this is possible by Proposition \ref{prop:limits}) and we assume that there exists $k^\pprime \in \mathbb N$ such that $\UUc_k \cap \llb \maxrho_{k-lq} + lq, \maxrho_k - M \rrb$ is either empty or an AP with difference $\delta$ for $k \ge k^\pprime$, then it becomes clear from the above that $\UUc_k \cap \llb k, \maxrho_k - M \rrb$ is an AP with difference $\delta$ for $k \ge \max(k^\prime, k^\pprime)$.

\medskip
\ref{it:main(2)} Note that $\mathbb N \cap L^\ast \ne \emptyset$, and set $n = \min (\mathbb N \cap L^\ast)$. Since $\rho < \infty$, we have by Proposition \ref{prop:accepter_elasticity} that
\begin{equation}
\label{equ:rescaling}
\sup L^\ast = \rho_n
\quad\text{and}\quad
\rho_{nk} = nk\rho\quad\text{for all}\quad k \in \mathbb N.
\end{equation}
Next, let $r$ be the smallest integer $> \rho \max \Delta$ such that $r \rho \in \NNp$ and $nr(\rho^2 - 1) \ge \frac{1}{\delta} \max \Delta$, which exists because $\rho$ is a rational number $> 1$.
We split the sequel of the proof into a series of claims.
\begin{claim}
\label{claim2}
Let $k$ be an integer $> \frac{1}{n} \rho \max \Delta$, and suppose that $kL^\ast \subseteq L$ for some $L \in \mathscr{L}$. Then $L = kL^\ast$ and $L$ is an AP with difference $\delta$, $\min L = nk$, and $\sup L = nk \rho$. In particular, $0 \notin L^\ast$.
\end{claim}
\begin{proof}[Proof of Claim \ref{claim2}]
Let $L \in \mathscr{L}$ such that $kL^\ast \subseteq L$. Then $nk \in L$, and since $nk > \rho \max \Delta$, we have by Lemma \ref{lem:AP_reduction_to_the_upper_part}\ref{item:AP_reduction(i)} that $\min L \ge 1$. On the other hand, $\rho < \infty$ implies by Proposition \ref{prop:limits} that $\maxrho_\kappa < \infty$ for all $\kappa \in \mathbb N$, whence every set in $\mathscr{L}$ is finite. Accordingly, $\min L \le nk$ and $k \sup L^\ast \le \sup L < \infty$, with the result that
$$
\frac{\sup(L)}{\min L} \le \rho = \rho(L^\ast) = \frac{1}{n}\sup L^\ast = \frac{k \sup L^\ast}{nk} \le \frac{\sup(L)}{\min L}.
$$
So we see that $\min L = nk = \min (kL^\ast)$ and $\sup L = k \sup L^\ast = \sup (kL^\ast)$. But $kL^\ast$ is the $k$-fold sumset of an AP with difference $\delta$, so it is itself an AP with difference $\delta$. It follows that $L = kL^\ast$, and hence $\sup L = nk \rho$ by \eqref{equ:rescaling}. The rest is clear.
\end{proof}
\begin{claim}
\label{claim3}
$\UUc_{nr\rho}$ is an AP with difference $\delta$, and in addition $\lambda_{nr\rho} = nr$ and $\rho_{nr\rho} = nr\rho^2$.
\end{claim}
\begin{proof}[Proof of Claim \ref{claim3}]
Since $r\rho$ is a positive integer and $\LLc$ is a directed family, there exist $L_1, L_2 \in \LLc$ such that $r L^\ast \subseteq L_1$ and $r\rho L^\ast \subseteq L_2$, which yields by \eqref{equ:rescaling} and Claim \ref{claim2} that $L_1 = r L^\ast$ and $L_2 = r\rho L^\ast$, and both $L_1$ and $L_2$ are APs with difference $\delta$. So we get by \eqref{equ:rescaling} that
$$
\sup L_1 = r \sup L^\ast = r \rho_n = nr \rho = r \rho \min L^\ast = \min L_2,
$$
which yields $L_1 \cup L_2 \subseteq \UUc_{nr\rho}$. Since $L_1 \cup L_2$ is an AP with difference $\delta$ and $\UUc_{nr\rho} \subseteq nr\rho + \delta \cdot \ZZb$ by Lemma \ref{lem_FrGe}\ref{it:lemma_FrGe(i)}, it is then enough for completing the proof of the claim to show that
$$
\lambda_{nr\rho} = \min L_1 = r \min L^\ast = nr
\quad\text{and}\quad
\rho_{nr\rho} = \sup L_2 = r\rho \sup L^\ast = nr\rho^2,
$$
as in that case $L_1 \cup L_2 = \UUc_{nr\rho}$. Now, the latter equality is clear, because $\rho_{nr\rho} = nr\rho^2$ by \eqref{equ:rescaling}. As for the former, suppose for a contradiction that $\min \UUc_{nr\rho} < nr$. Since (by construction) $nr \ge 1+ \max \Delta$, there then exists $L \in \LLc_{nr\rho}$ such that $L^+ \cap \llb 1, nr-1 \rrb$ is non-empty, with the result that $\rho(L^+) \ge nr\rho/\min L^+ > \rho$, which is, however, impossible.
\end{proof}
\begin{claim}
\label{claim4}
Given $j \in \llb 0, n - 1 \rrb$, there exists $m_j \in \NNp$ such that the following conditions hold:
\begin{enumerate}[label={\rm (\roman{*})}]
\item\label{it:cond(i)_claim4} $\rho_{m_jn+j} - m_jn \rho = \max_{\kappa \ge 1} (\rho_{n\kappa+j} - n\kappa\rho)$.
\item\label{it:cond(ii)_claim4} $\rho_{(m_j+i)n+j} = \rho_{m_jn+j} + in\rho$ for all $i \in \NNb$.
\end{enumerate}
\end{claim}
\begin{proof}[Proof of Claim \ref{claim4}]
By Proposition \ref{prop:limits}, $\rho_\kappa \le \kappa \rho$ for every $\kappa \in \NNp$, and hence $\rho_{n\kappa + j} -  n\kappa \rho \le j\rho \le (n-1) \rho$ for all $\kappa \in \NNp$.
Since $\rho < \infty$, this gives that the set
$$
\mathcal D_j = \{\rho_{n\kappa + j} -  n\kappa\rho: \kappa \in \NNp\} \subseteq \mathbb Z
$$
has a maximum element; let $m_j \in \NNp$ such that $\rho_{m_jn + j} -  m_jn\rho = \max \mathcal D_j$. Accordingly, we get by \eqref{equ:rescaling} and Lemma \ref{prop:basic(1)}\ref{it:prop:basic(1)(ii)} that, for every $i \in \NNb$,
\begin{equation*}
\label{equ:last_equ}
\begin{split}
\rho_{(m_j+i)n+j} - (m_j + i)n\rho
    & \le \rho_{m_jn+j} - m_jn \rho  \le \rho_{m_jn+j} + \rho_{in} - (m_j + i)n \rho \\
    & \le \rho_{(m_j+i)n+j} - (m_j + i)n \rho,
\end{split}
\end{equation*}
which leads to the desired conclusion.
\end{proof}

At this point, let $k^\ast$ be a fixed integer $> m_0 + \cdots + m_{n-1} + nr\rho$, where for each $j \in \llb 0, n - 1\rrb$ we denote by $m_j$ the smallest positive integer for which conditions \ref{it:cond(i)_claim4} and \ref{it:cond(ii)_claim4} in Claim  \ref{claim4} are verified. Then pick any integer $k \ge k^\ast$. By construction and \eqref{equ:rescaling}, we have
$$
\sup \UUc_k = \rho_k = \rho_{k-nr\rho} + nr\rho^2 = \sup(\UUc_{nr\rho}) + \sup(\UUc_{k-nr\rho}) = \sup(\UUc_{nr\rho} + \UUc_{k-nr\rho}).
$$
On the other hand, we get by Lemma \ref{prop:basic(1)}\ref{it:prop:basic(1)(ii)} that $\UUc_{nr\rho} + \UUc_{k-nr\rho} \subseteq \UUc_k$, and by Claim \ref{claim3} that $\UUc_{nr\rho}$ is an AP with difference $\delta$. Moreover, Claim \ref{claim3} gives that
$$
\sup \UUc_{nr\rho} - \min \UUc_{nr\rho} \ge nr(\rho^2 - 1) \ge \frac{1}{\delta}\max \Delta,
$$
while $\Delta(\UUc_{k-nr\rho}) \subseteq \delta \cdot \NNb$ and $\sup \Delta(\UUc_{k-nr\rho}) \le \sup \Delta$ by Lemma \ref{lem_FrGe}. It follows that also the sumset $\UUc_{nr\rho} + \UUc_{k - nr\rho}$ is an AP with difference $\delta$, and so is $\UUc_k \cap \llb k, \rho_k \rrb$, as we see from the above that $\UUc_k \cap \llb k, \rho_k \rrb = \bigl(\UUc_{nr\rho} + \UUc_{k - nr\rho}\bigr) \cap \llb k, \rho_k \rrb$. Therefore, we get by Lemma \ref{lem:AP_reduction_to_the_upper_part} that also $\UUc_k$ is an AP with difference $\delta$, which completes the proof of the theorem, as $k$ is any integer $\ge k^\ast$.
\end{proof}
\begin{proof}[Proof of Corollary \ref{cor:generalization}]
By hypothesis, there exist $k_0, d \in \NNp$ and $M \in \NNb$ such that, for every $k \ge k_0$, $\UUc_k$ is an AAP with difference $d$ and bound $M$, with the result that the set
$\UUc_k^\ast = \UUc_k \cap \llb \lambda_k + M, \maxrho_k - M \rrb
$
is an AP with difference $d$. Moreover, we have that $k \in \UUc_k$ for all $k \in \NNp$, and by Lemma \ref{lem:reduction_to_the_upper_part} we can assume $d = \delta$, so that $\UUc_k^\ast \subseteq k + \delta \cdot \ZZb$.

\ref{it:cor:some_properties_of_ASTU_families(i)}
Suppose now that $\maxrho_k < \infty$ for all $k \in \NNp$. Then, for every $k \in \NNp$ we can write
$$
\UUc_k = \bigl( \llb \lambda_k, k-1 \rrb \cap \UUc_k
\bigr) \cup \{k\} \cup \bigl( \llb k+1, \maxrho_k \rrb \cap \UUc_k
\bigr),
$$
which shows, together with the preliminary considerations, that all of $\UUc_1, \UUc_2, \ldots$ are AAPs with difference $d$ and bound
$\max \{M, M_0\}$, where $M_0 = \sup_{1 \le k < k_0} \max \{k - \lambda_k, \maxrho_k - k\}$.

\ref{it:cor:some_properties_of_ASTU_families(ii)} If there is some $k \in \NNp$ such that $\maxrho_k =
\infty$, then each side of equation \eqref{equ:limit_of_size} is infinite, and we are done. Otherwise, we have by the preliminary considerations that, for $k \ge k_0$,
$$
\frac{(\maxrho_k - M) - (\lambda_k +
M) - \delta}{k\delta} \le \frac{|\UUc_k|}{k} \le \frac{\maxrho_k - \lambda_k + \delta}{k\delta},
$$
which, combined with Proposition \ref{prop:limits}, leads, in the limit as $k \to \infty$, to the desired conclusion that the rightmost and leftmost sides of \eqref{equ:limit_of_size} are equal.

As for the rest,  if $X$ and $Y$ are non-empty sets of integers, then $|X+Y| \ge |X| + |Y| - 1$ (see, e.g., \cite[Theorem 3.1]{Gr13a}). This, along with Lemma \ref{prop:basic(1)}, yields that the sequence $\bigl(|\UUc_k| - 1\bigr)_{k \ge 1}$ is superadditive, so we are done by Fekete's lemma.
\end{proof}

\medskip
\section{Arithmetic of commutative monoids} \label{3}
\medskip

In this section we apply the results on directed families from Section \ref{2} (in particular, Theorem \ref{th:structure_theorem} and Corollary \ref{cor:generalization}) to systems of sets of lengths of certain classes of semigroups. The main abstract results  are Theorems \ref{3.4} and \ref{3.6}, and they are valid for a class of commutative semigroups which are not necessarily cancellative. These theorems as well as the preparatory lemmas are well-known in the cancellative setting, partly with different bounds. However, there are several striking differences between the present setting and the cancellative one, which are gathered in Remarks \ref{important}. The semigroups we have in mind are certain semigroups of ideals and of modules. They will be discussed in Subsections \ref{subsec:3.2} and \ref{subsec:3.3} and  have been the motivating examples for this paper (Corollaries \ref{3.8} and \ref{3.9}).

We first introduce a suitable class of semigroups and give some motivation for it.
A \textit{semigroup} $H$ will be a non-empty set  together with an associative and commutative binary operation  such that $H$ possesses an identity element, which we denote by $1_H$ (or simply by $1$) if $H$ is written multiplicatively. Apart from  semigroups of modules (discussed in Subsection \ref{subsec:3.3}) and from power monoids (introduced in Subsection \ref{subsec:3.4}), we will always adopt the multiplicative notation. We use $H^{\times}$ for the group of  invertible elements of $H$, and if $H$ is cancellative, then we denote by ${\sf q}(H)$ the quotient group of $H$.

Suppose that $H$ has the property that every  non-unit can be written as a finite product of distinguished elements. Our goal is  to count the number of factors in such a factorization, and then to attach to each element $a \in H$ a  set of factorization lengths $\LLs (a) \subseteq \NNb$. This should be done in a way that the product of two factorizations of elements $a$ and $b$ is a factorization of $ab$, which implies that  $\LLs (a) + \LLs (b) \subseteq \LLs (ab)$. Every unit $u \in H$ is considered to have an empty factorization, and hence we define  $\LLs (u)=\{0\}$. So if $a$ is a non-unit and $b$ is any element of the semigroup such that $a = ab$, then $\LLs (a) + \LLs (b) \subseteq \LLs (a)$. In particular, if $\LLs (a)$ is a finite non-empty set, this is possible only if $\LLs (b) = \{0\}$, and hence $b$ has to be a unit.

Thus a theory of sets of factorization lengths subjected to
the above constraints is necessarily restricted to  semigroups satisfying one of the next two equivalent conditions:
\begin{itemize}
\item  If $a, u \in H$ and $a = au$, then $u \in H^{\times}$.

\item If $a, b, u, v  \in H$ are such that $a = bu$ and $b = a v$, then $u, v \in H^{\times}$.
\end{itemize}
We refer to such semigroups as {\it unit-cancellative}, and for convenience we adopt the following convention:
\begin{center}
{\it
Throughout, a
monoid will always mean a commutative unit-cancellative semigroup.}
\end{center}
Let $H$ be a monoid.
Given $a, b \in H$, we say that  $a$ and  $b$ are associated (and write $a \simeq b$) if $aH^{\times} = bH^{\times}$, which, by unit-cancellativity, is equivalent to $aH = bH$. This defines a congruence on $H$, and $H_{\red} = H/\simeq$ denotes the associated reduced monoid of $H$. The monoid $H$ is called reduced if $a,b \in H$ and $a \simeq b$ yield $a=b$.
We say that $a$ divides $b$ (and write $a \t b$) if $b \in aH$.  A  non-unit element $p \in H$ is said to be:
\begin{itemize}
\item \textit{irreducible} (or an \textit{atom}) if $a, b \in H$ and $p = ab$ imply $a \in H^{\times}$ or $b \in H^{\times}$;
\item \textit{prime} if $a, b \in H$ and $p \t ab$ imply $p \t a$ or $p \t b$.
\end{itemize}
The prime elements of a monoid are irreducible. We denote by $\mathcal A (H)$ the set of atoms of $H$.
We say that $H$ is \textit{atomic} if every  non-unit is a finite product of atoms.
For a set $P$, we denote by $\mathcal F (P)$ the free abelian monoid with basis $P$. An element $a \in \mathcal F (P)$ will be written as $a = \prod_{p \in P} p^{\mathsf v_p (a)}$, where $\mathsf v_p (a) \in \mathbb{N}_0$ for all $p \in P$ and $\mathsf v_p (a) = 0$ for almost all $p \in P$, and we will denote by $|a|= \sum_{p \in P} \mathsf v_p (a) \in \mathbb{N}_0$ the length of $a$. Clearly, $P$ is the set of atoms of $\mathcal F (P)$ and every atom is prime.

A monoid $H$ is said to be a \textit{Krull monoid} if it is cancellative and one of the following equivalent conditions holds:
\begin{enumerate}
\item[(a)] $H$ is completely integrally closed and $v$-noetherian.

\item[(b)] There are a free abelian monoid $F$ and a homomorphism $\varphi: H \to F$, called a \textit{divisor homomorphism}, with the property that if $a, b \in H$ and $\varphi (a) \t \varphi (b)$ (in $F$), then $a \t b$ (in $H$).
\end{enumerate}
The theory of Krull monoids is presented in the monographs \cite{HK98, Ge-HK06a}. We recall that an integral domain is a Krull domain if and only if its monoid of non-zero elements is Krull. Thus Condition (a) shows that every noetherian integrally closed domain is Krull (for more examples we refer to \cite{Ge16c}).

It is well-known that the ascending chain condition (shortly, ACC) on principal ideals implies atomicity (in various classes of rings and semigroups). For the sake of completeness we provide a short proof of this and some related facts in the present setting.

\medskip
\begin{lemma} \label{3.1}
Let $H$ be a monoid.
\begin{enumerate}[label={\rm (\arabic{*})}]
\item\label{it:lemma3.1(1)} If $H$  satisfies the {\rm ACC} on principal ideals, then $H$ is atomic.
\item\label{it:lemma3.1(2)} If $H$ is atomic, then for every set $P$ of pairwise non-associated primes of $H$ there is a submonoid $T \subseteq H$ such that $H = \mathcal F (P)  T$.
\end{enumerate}
\end{lemma}

\begin{proof}
\ref{it:lemma3.1(1)} Assume to the contrary that the set
\[
\Omega = \{aH \colon a \in H  \ \text{is not a product of finitely many atoms} \}
\]
is non-empty. Then it has a maximal element, say $aH$. Since $a$ is not an atom, there are non-units $b,c \in H$ such that $a=bc$. Thus we obtain that $aH \subseteq bH$ and $aH \subseteq cH$. Assume for a contradiction that $aH=bH$. Then there is an element $d \in H$ such that $b=ad$, and unit-cancellativity implies that $c \in H^{\times}$, a contradiction. Thus $aH \subsetneq bH$, and similarly $aH \subsetneq cH$. It follows by the maximality of $aH$ that $b$ and $c$ are finite products of atoms, and hence so is $a$, a contradiction.

\smallskip
\ref{it:lemma3.1(2)} Let $T \subseteq H$ be the set of all $a \in H$ such that $ p \nmid a$ for all $p \in P$. Then $H^{\times} \subseteq T$ and $T$ is a submonoid of $H$. Since $H$ is atomic, it follows that $H = \mathcal F (P)T$.
\end{proof}

Suppose that $H$ is an atomic monoid. If $a = u_1 \cdot \ldots \cdot u_k$ is a product of $k$ atoms, then $k$ is called the length of the factorization and the set $\mathsf L_H (a) = \mathsf L (a) \subseteq \N$ of all possible factorization lengths is called the \textit{set of lengths} of $a$. It is convenient to take $\mathsf L (u)=\{0\}$ for all units $u \in H^{\times}$. Then
\[
\LLc (H) = \{\mathsf L (a) \colon a \in H \}
\]
denotes the \textit{system of sets of lengths} of $H$.
If $H \ne H^{\times}$, $\LLc (H)$ is a  directed family of subsets of $\N_0$, and we set $\Delta (H) = \Delta \bigl( \LLc (H) \bigr), \ \rho (H) = \rho \bigl( \LLc (H) \bigr)$, and for every $k \in \N$,
\[
 \UUc_k (H) = \UUc_k \bigl( \LLc (H) \bigr), \quad  \rho_k (H) = \rho_k \bigl( \LLc (H) \bigr), \quad \text{and} \quad \lambda_k (H) = \lambda_k \bigl( \LLc (H) \bigr) \,.
\]
If $H = H^{\times}$, then for convenience we set $\Delta (H)=\emptyset$, $\rho (H)=1$, and for every $k \in \mathbb N$, $\mathscr{U}_k(H) = \{k\}$ and $\rho_k(H) = \lambda_k(H) = k$.
We say that $H$ is a \textit{\BF-monoid} if every   $L \in \LLc (H)$ is finite, and that $H$ has \textit{accepted elasticity} if $\rho(H) = \rho(L) < \infty$ for some $L \in \mathscr{L}(H)$. All examples discussed in this paper are BF-monoids, and Proposition \ref{3.5} provides a sufficient condition for a monoid to be a BF-monoid.

With this in mind, we now recall a couple more of arithmetical concepts, that is the catenary and tame degrees and the $\omega$-invariants of a monoid.

To begin, we denote by $\mathsf Z (H) = \mathcal F (\mathcal A (H_{\red}))$ the factorization monoid of $H$, and by $\pi \colon \mathsf Z (H) \to H_{\red}$, $u \mapsto u$ for all $u \in \mathcal A (H_{\red})$,  the canonical epimorphism. For every  $a \in H$, we let
\[
\mathsf Z_H (a) = \mathsf Z (a) = \pi^{-1} (aH^{\times}) \subseteq \mathsf Z (H)
\]
be the \textit{set of factorizations} of $a$, and  we have $\mathsf L (a) = \fixed[-0.3]{\text{ }}\bigl\{ |z| \colon z \in \mathsf Z (a) \bigr\} \fixed[-0.3]{\text{ }} \subseteq \N_0$.

Let $z, z' \in \mathsf Z (H)$. Then there exist $\ell, m,n \in \N_0$ and $u_1, \ldots, u_{\ell}, v_1, \ldots, v_m, w_1, \ldots, w_n \in \mathcal A (H_{\red})$ with $\{v_1, \ldots, v_m\} \cap \{w_1, \ldots, w_n\} = \emptyset$ such that
\[
z= u_1 \cdot \ldots \cdot u_{\ell}v_1 \cdot \ldots \cdot v_m \quad \text{and} \quad z' = u_1 \cdot \ldots \cdot u_{\ell}w_1 \cdot \ldots \cdot w_n \fixed[0.15]{\text{ }}.
\]
We call $\mathsf d (z,z') = \max \{m,n\}$ the \textit{distance} between $z$ and $z'$.
Given $a \in H$, the \textit{catenary degree} $\mathsf c (a)$ of $a$ is then the smallest $N \in \N_0 \cup \{\infty\}$ with the following property{\rm \,:}
\begin{itemize}
\item[] For any two factorizations $z,z' \in \mathsf Z (a)$ there are factorizations $z_0, \ldots, z_k$ of $a$ with $z_0 = z$ and $z_k = z'$ such that $\mathsf d (z_{i-1}, z_i) \le N$ for every $i \in \llb 1,k \rrb$.
\end{itemize}
Accordingly, we define the \textit{catenary degree} of $H$ by
\[
\mathsf c (H) = \sup \{ \mathsf c (a) \colon  a \in H \} \in \N_0 \cup \{\infty\}.
\]
Next for every subset $I \subseteq H$, we let $\omega (H,I)$ be the smallest $N \in \N_0\cup \{\infty\}$ for which the following holds{\rm \,:}
    \begin{itemize}
    \item[] If $n \in \N$ and $a_1, \ldots, a_n \in H$ with $a_1 \cdot \ldots \cdot a_n \in I$, then there exists a subset $\Omega \subseteq \llb 1,n \rrb$  such that $|\Omega| \le N$ and $\prod_{\nu \in \Omega} a_{\nu} \in I$.
    \end{itemize}
In particular, for every $a \in H$ we let $\omega (H,a) = \omega (H, aH)$, and we set
      \[
      \omega (H) = \sup \{ \omega (H, u) \colon u \in \mathcal A (H) \} \in \N_0 \cup \{\infty\} \,.
      \]
By convention, empty products are equal to $1 \in H$. Thus, if $a  \in H^{\times}$, then $\omega (H,a)=0$, and if $p \in H$ is a prime, then $\omega (H,p)=1$.

Lastly, for every $u \in \mathcal A (H_{\red})$ we let  $\mathsf t (H,u)$ denote the smallest $N \in \N_0\cup \{\infty\}$ such that:
\begin{itemize}
\item[] Given $a \in H$ with $\mathsf Z (a) \cap u\mathsf Z (H) \ne \emptyset$ and $z \in \mathsf Z (a)$, there exists a factorization $z' \in \mathsf Z (a) \cap u\mathsf Z (H)$ with $\mathsf d (z,z')\le N$.
\end{itemize}
We say that $H$ is \textit{locally tame} if $\mathsf t (H,u)<\infty$ for all $u \in \mathcal A (H_{\red})$, and \textit{\textup{(}globally\textup{)} tame} if its \textit{tame degree}
\[
\mathsf t (H) = \sup \{\mathsf t (H,u) \colon u \in \mathcal A (H_{\red}) \} \in \N_0 \cup \{\infty\}
\]
is finite. In the next proposition we collect
some elementary properties of the above invariants that
are well-known for cancellative monoids (\cite{Ge-Ha08a, Ge-Ka10a}).

\medskip
\begin{proposition} \label{3.2}
Let $H$ be an atomic monoid.
\begin{enumerate}[label={\rm (\arabic{*})}]
\item\label{it:prop3.2(1)} For every $a \in H $ we have $\sup \mathsf L (a) \le \omega (H,a)$. In particular, if $\omega (H,a) < \infty$ for every $a \in H$, then $H$ is a \BF-monoid.

\smallskip
\item\label{it:prop3.2(2)} For every $a \in H$ with $|\mathsf Z (a)|\ge 2$, we have $1 + \sup \Delta ( \mathsf L (a)) \le \mathsf c (a)$.  In particular, if $\Delta (H) \ne \emptyset$, then $1 + \sup \Delta (H) \le \mathsf c (H)$.

\smallskip
\item\label{it:prop3.2(3)} For every $u \in \mathcal A (H)$ with $\omega (H, u)< \infty$ we have $\mathsf t (H, u) \le \rho_{\omega (H,u)} (H)$.
\end{enumerate}
\end{proposition}

\begin{proof}
We may suppose that $H$ is reduced.

\ref{it:prop3.2(1)}  Let $a \in H$. If $\omega (H,a)=\infty$, the assertion is trivial.  Suppose that $\omega (H,a)<\infty$ and assume for a contradiction that $\sup \mathsf L (a) > \omega (H,a)$. Then there exist $m \in \N$ with $m > \omega (H,a)$ and atoms $v_1, \ldots, v_{m}$ such that $a=v_1 \cdot \ldots \cdot v_{m}$. Accordingly,  there is a subset $\Omega \subseteq \llb 1,m \rrb$, say $\Omega = \llb 1,\ell \rrb$ with $\ell \le \omega (H,a) < m$, such that $a$ divides $v_1 \cdot \ldots \cdot v_{\ell}$. Then unit-cancellativity implies that $v_{\ell+1} \cdot \ldots \cdot v_m \in H^{\times}$, a contradiction.

\smallskip
\ref{it:prop3.2(2)} We start by verifying a basic inequality.  Let $a\in H$, and let  $z, z' \in \mathsf Z (a)$ be distinct factorizations. Set $x = \gcd (z,z')$, and $z=xy$, $z'=xy'$ with $y,y' \in \mathsf Z (H)$. Then $\mathsf d (z,z')=\max \fixed[-0.3]{\text{ }}\bigl\{ |y|, |y'| \bigr\}\fixed[-0.25]{\text{ }}$, and unit-cancellativity implies that $|y|\ge 1$ and $|y'|\ge 1$.  Thus it follows that
\begin{equation*} \label{basic-equ}
1+ \big| |z|-|z'| \big| = 1 + \big| |y|-|y'| \big| \le \max \fixed[-0.3]{\text{ }}\bigl\{ |y|, |y'| \bigr\}\fixed[-0.25]{\text{ }} = \mathsf d (z,z') \,.
\end{equation*}
Now suppose that $a \in H$ with $\Delta (\mathsf L (a)) \ne \emptyset$ and let $s \in \Delta (\mathsf L (a))$. We have to show that $1+s\le \mathsf c (a)$. Indeed, there exist $z, z' \in \mathsf Z (a)$ such that $|z'|=|z|+s$, but there is no factorization $z'' \in \mathsf Z (a)$ with $|z|<|z''|<|z'|$. On the other hand, there are factorizations $z_0, \ldots, z_k$ of $a$ with $z_0 = z$ and $z_k=z'$ such that $\mathsf d (z_{i-1},z_i) \le \mathsf c (a)$ for every $i \in \llb 1,k \rrb$. Thus there is some $\nu \in \llb 1,k \rrb$ such that $|z_{\nu-1}|\le |z|$ and $|z_{\nu}|\ge |z'|$, and hence we get from the above that
\[
1+s \le 1 + |z_{\nu}|-|z_{\nu-1}| \le \mathsf d (z_{\nu-1}, z_{\nu})\le \mathsf c (a) \,.
\]
Finally, suppose that $\Delta (H)\ne \emptyset$. Then there is an element $a\in H$ with $|\mathsf Z (a)|\ge |\mathsf L (a)|\ge 2$, and for all $a\in H$ with $|\mathsf Z (a)|\ge 2$ we have $1 + \sup \Delta ( \mathsf L (a)) \le \mathsf c (a) \le \mathsf c (H)$. So the assertion follows.

\smallskip
\ref{it:prop3.2(3)} Pick $u \in \mathcal A(H)$ with $\omega (H,u) < \infty$, and let $a \in H$ with $\mathsf{Z}(a) \cap u \mathsf Z (H) \ne \emptyset$. If $z = v_1 \cdot \ldots \cdot v_n \in \mathsf{Z}(a)$, then $u \mid v_1 \cdot \ldots \cdot v_n$ (in $H$), so (after renumbering if necessary) there exists $k \in \llb 1, n \rrb$ with $k \le \omega(H, u)$ such that $u \mid v_1 \cdot \ldots \cdot v_k$. Thus $v_1 \cdot \ldots \cdot v_k = u\fixed[0.15]{\text{ }} u_2 \cdot \ldots \cdot u_\ell$ for some $u_2, \ldots, u_\ell \in \mathcal A(H)$ with
$\ell \le \rho_k(H) \le \rho_{\omega(H, u)}(H)$.
It follows that  $z^\prime = u\fixed[0.15]{\text{ }} u_2 \cdot \ldots \cdot u_\ell \cdot v_{k+1} \cdot \ldots \cdot v_n  \in \mathsf{Z}(a)$ and $
\mathsf d(z, z^\prime) \le \max\{k, \ell\} \le \rho_{\omega(H, u)}(H)$. Therefore $\mathsf t (H,u) \le \rho_{\omega(H, u)}(H)$.
\end{proof}

\smallskip
Next we study the $\omega$-invariants with ideal theoretic tools.
A \textit{weak ideal system} on a monoid $H$ is a map $r \colon \PPc (H) \to \PPc (H), X \mapsto X_r$ such that the following conditions are satisfies for all subsets $X, Y \subseteq H$ and all $c \in H$:
\begin{itemize}
\item $X \subseteq X_r$.

\item $X \subseteq Y_r$ implies $X_r \subseteq Y_r$.

\item $c H \subseteq \{c\}_r$.

\item $cX_r \subseteq (cX)_r$.
\end{itemize}
An \textit{ideal system} on $H$ is a weak ideal system such that $cX_r = (cX)_r$ for all $X \subseteq H$ and all $c \in H$. We refer to \cite{HK98} for a thorough treatment of ideal systems.

Let $r$ be a weak ideal system. A subset $I \subseteq H$ is called an $r$-ideal if $I_r=I$. We denote by $\mathcal I_r (H)$ the set of all non-empty $r$-ideals, and we define $r$-multiplication by setting $I \cdot_r J = (IJ)_r$ for all $I,J \in \mathcal I_r (H)$. Then $\mathcal I_r (H)$ together with $r$-multiplication is a reduced semigroup with identity element $H$.
An $r$-ideal $\mathfrak p \in \mathcal I_r (H)$ is called prime if $\mathfrak p \ne H$, and $a,b \in H$ and  $ab \in \mathfrak p$ imply $a \in \mathfrak p$ or $b \in \mathfrak p$. We denote by $r$-$\spec (H)$ the set of all prime $r$-ideals of $H$. The monoid $H$ is called $r$-noetherian if it satisfies the ACC on $r$-ideals. For $I \in \mathcal I_r (H)$ and $h \in H$, we set $(I \colon h) = \{ b \in H \colon hb \in I\}$ and note that $(I \colon h) \in \mathcal I_r (H)$  (see \cite[Section 2.4]{HK98}).

The result of the next proposition was first established for the $v$-system (i.e., the system of divisorial ideals) with a different proof (see \cite[Theorem 4.2]{Ge-Ha08a}).

\medskip
\begin{proposition} \label{3.5}
Let $H$ be a monoid and let $r$ be  an $r$-noetherian weak ideal system on $H$.
\begin{enumerate}[label={\rm (\arabic{*})}]
\item\label{it:prop_3.5(1)} For every $I \in \mathcal I_r (H)$ we have $\omega (H,I)<\infty$.

\smallskip
\item\label{it:prop_3.5(2)} If principal ideals are $r$-ideals, then $H$ is a \BF-monoid.
\end{enumerate}
\end{proposition}

\begin{proof}
\ref{it:prop_3.5(1)} If $I \subseteq H$ is a prime ideal, then $\omega (H, I)=1$. We start with the following two assertions.

\smallskip
\begin{enumerate}
\item[({\bf A1})] Given $I \in \mathcal I_r (H)$, there exists $J \in \mathcal I_r (H)$ with $I \subsetneq J$ such that the set $(I \colon h)$ is a prime ideal for every $h \in J \setminus I$.

\item[({\bf A2})] If $I, J \in \mathcal I_r (H)$ are as above, then $\omega (H,I) \le \omega (H,J)+1$.
\end{enumerate}

\smallskip
\noindent
\textit{Proof of} ({\bf A1}). Consider the set
\[
\Omega = \big\{ (I \colon h)  \colon h \in H \setminus I \big\} \subseteq \mathcal I_r (H) \fixed[0.1]{\text{ }},
\]
and choose an element $h_0 \in H \setminus I$ such that   $\mathfrak p = (I \colon h_0) \in \Omega$ is maximal. We claim that $\mathfrak p$ is a prime ideal. In fact, assume  that there exist $c,d \in H$ with $cd \in \mathfrak p$, but $d \notin \mathfrak p$. Then $h_0cd \in I$, $h_0d \notin I$, and
\[
\mathfrak p = (I \colon h_0) \subsetneq (I \colon h_0c) \fixed[0.1]{\text{ }} \,.
\]
Thus, the maximality of $\mathfrak p$ yields that $h_0c \in I$, and hence $c \in \mathfrak p$. Now, we assert that the ideal
\[
J = ( I \cup \{h_0\})_r
\]
has the required properties. Indeed, let $h \in J \setminus I$, and let $b \in H$ with $h_0b \in I$. Since $Ib \subseteq I$, we have
\[
hb  \in Jb = ( I \cup \{h_0\})_r b \subseteq (bI \cup \{h_0b\})_r \subseteq I_r = I \,.
\]
This shows that
\[
\mathfrak p = (I \colon h_0) \subseteq (I \colon h),
\]
which, again by the maximality of $\mathfrak p$, implies that $\mathfrak p = (I \colon h)$ is a prime ideal.
\qed ({\bf A1})

\smallskip
\noindent
\textit{Proof of} ({\bf A2}). Let $n \in \N$ and $a_1, \ldots, a_n \in H$ with $a_1 \cdot \ldots \cdot a_n \in I$. Then $a_1 \cdot \ldots \cdot a_n \in J$, and there is a subset $\Omega \subseteq \llb 1,n\rrb$, say $\Omega = \llb 1,m \rrb$, with $m \le \omega (H,J)$ and $a_1 \cdot \ldots \cdot a_m \in J$.
If $a_1 \cdot \ldots \cdot a_m \in I$, then we are done.
Otherwise,
$a_{m+1} \cdot \ldots \cdot a_n \in (I \colon a_1 \cdot \ldots \cdot a_m) = \mathfrak p$, and since $\mathfrak p$ is a prime ideal, there exists $\nu \in \llb m+1,n \rrb$, say $\nu=m+1$, such that $a_{m+1} \in \mathfrak p$. Clearly, this implies that $a_1 \cdot \ldots \cdot a_{m+1} \in I$, and hence $\omega (H,I) \le \omega (H,J)+1$. \qed ({\bf A2})

Let $I \in \mathcal I_r (H)$. Since $H$ is $r$-noetherian, we get from ({\bf A1}) that there is an ascending chain of $r$-ideals $$
I=I_0 \subsetneq I_1 \subsetneq I_2 \subsetneq \ldots \subsetneq  I_n = H
$$
such that $(I_{\nu-1} : h)$ is a prime ideal for all $\nu \in \llb 1,n \rrb$ and $h \in I_\nu \setminus I_{\nu - 1}$. So, we conclude by ({\bf A2}) that
\[
\omega (H,I) \le \omega (H, I_1) +1 \le \omega (H, I_2)+2 \le \ldots \le \omega (H,I_n)+n=n \,.
\]

\smallskip
\ref{it:prop_3.5(2)} Since $H$ is $r$-noetherian, $H$ satisfies the ACC on principal ideals, and hence it is atomic by Lemma \ref{3.1}. Thus, the assertion follows by point \ref{it:prop_3.5(1)} and Proposition \ref{3.2}\ref{it:prop3.2(1)}.
\end{proof}

\smallskip
In the next proposition we will need the $s$-system of a monoid. Indeed, if $H$ is a monoid, then the
$s$-system $s \colon \PPc (H) \to \PPc (H)$, defined by $X_s = XH$ for all $X \subseteq H$, is an ideal system on $H$.

\medskip
\begin{proposition} \label{3.3}
Let $H$ be a monoid such that $H_\red$ is finitely generated. Then $H$ is an $s$-noetherian \BF-monoid, and the following hold:
\begin{enumerate}[label={\rm (\arabic{*})}]

\item\label{it:3.3(1)} $\rho(H) \in \mathbb Q$ and there exists $K \in \mathbb N$ such that $\rho_{k+1}(H) \le \rho_k(H) + K < \infty$ for all $k \in \mathbb N$.

\item\label{it:3.3(2)} The set of distances  $\Delta (H)$ is finite.

\item\label{it:3.3(3)} $\omega (H) < \infty$ and \ $\mathsf t(H) < \infty$.
\end{enumerate}
\end{proposition}
\begin{proof}
The monoid $H$ is $s$-noetherian by  \cite[Theorem 3.6]{HK98} and hence it is a BF-monoid by Proposition \ref{3.5}. We are left to verify the assertion on the arithmetical invariants, and to do so we may
assume without restriction that $H$ is reduced and $\mathcal A (H) = \{u_1, \ldots, u_s\}$ is non-empty. We will denote by $\preceq$ the product order induced on $\mathbb N_0^s$ by the usual order of $\mathbb N_0$.

\smallskip
\ref{it:3.3(1)} Let $\pi \colon \mathsf Z (H)  \to H$ be the canonical epimorphism and set $S^* = S \setminus \{(1,1)\}$, where
\[
S = \{(x,y) \in \mathsf Z (H) \times \mathsf Z (H) : \pi (xz) = \pi
(yz) \text{ for some }z \in \mathsf Z (H)\}\,.
\]
Pick $(x,y) \in S^*$ with $x \ne 1$, and suppose for a contradiction that $y = 1$. Then $\pi(y) = 1$, and letting $z \in \mathsf Z (H)$ such that $\pi (xz)=\pi (yz)$ yields $\pi (z) \fixed[0.1]{\text{ }} \pi (x) = \pi (z)$. But this implies, together with unit-cancellativity, that $\pi (x)=1$, and hence $x=1$, a contradiction.

It follows that $|x|\ge 1$ and $|y|\ge 1$ for all $(x,y) \in S^*$, and accordingly we can define
$$
\rho^\ast(H) = \sup \fixed[-0.25]{\text{ }}\bigl\{ |y|^{-1}|x|: (x,y) \in S^*
\bigr\}.
$$
We note that $\rho(H) \le \rho^\ast(H)$, and we want to show that there exists $(\bar x,\bar y) \in S^*$ such that
$\rho^\ast(H) = |\bar y|^{-1}|\bar x|$.
To this end, consider the homomorphism
$$
f: \mathsf Z (H) \times \mathsf Z (H) \to (\N_0^s \times \N_0^s, +),
\left(\prod_{i=1}^s u_i^{m_i}, \prod_{i=1}^s u_i^{n_i}\right) \mapsto \bigl( (m_i)_{i=1}^s, (n_i)_{i=1}^s \bigr).
$$
In fact, $f$ is an isomorphism, and we have by Dickson's Theorem \cite[Theorem 1.5.3]{Ge-HK06a} that the set $f (S^*)$ has only finitely many minimal points relative to the order $\preceq$. Let $T \subseteq S^*$ be the inverse image of the set of these minimal points. It is sufficient to prove that
\[
\frac{|x|}{|y|} \le \max \fixed[-0.75]{\text{ }}\left\{ \frac{|x'|}{|y'|}:
(x',y') \in T \right\} \quad \text{for all} \quad (x,y) \in
S^*\,.
\]
For this, pick $(x,y) \in S^\ast$. We proceed by induction on $|x| + |y|$. If $(x,y) \in T$, then there is
nothing to do. So suppose that  $(x,y) \notin T$.  Then there exists $(x_1, y_1) \in T$ such that $(x,y) = (x_1 x_2,\, y_1
y_2)$ for some $(x_2, y_2) \in \mathsf Z (H) \times \mathsf Z (H)$. We claim that $(x_2, y_2) \in S^*$.

Indeed, let $z, z_1 \in \mathsf Z (H)$ such that $\pi (xz)=\pi (yz)$ and $\pi (x_1z_1)=\pi (y_1z_1)$, and set $z_2= x_1y_1z_1z$. Then $x_2z_2=xz y_1z_1$ and $y_2z_2= yz x_1z_1$,
which implies that $\pi (x_2z_2)=\pi (y_2z_2)$, and hence $(x_2, y_2) \in S^\ast$.

It is thus clear that $|x_j| + |y_j| < |x| + |y|$ for $j \in \{1,2\}$, and  by the induction hypothesis we have that
\[
\frac{|x|}{|y|} = \frac{|x_1|+|x_2|}{|y_1|+|y_2|} \le \max \fixed[-0.75]{\text{ }}\left\{
\frac{|x_1|}{|y_1|},\, \frac{|x_2|}{|y_2|}\right\} \le \max\fixed[-0.75]{\text{ }} \left\{
\frac{|x'|}{|y'|}: (x',y') \in T \right\},
\]
as was desired.

Now, based on the above, let $(\bar x, \bar y) \in T$ such that $\rho^\ast(H) = |\bar y|^{-1} |\bar x|$,
and let $\bar{z} \in \mathsf Z (H)$ such that $\pi (\bar{x} \fixed[0.2]{\text{ }} \bar{z})= \pi (\bar{y} \fixed[0.2]{\text{ }} \bar{z})$. Then, for every $i \in \mathbb N$ we have $\pi (\bar{x}^i \bar{z}) = \pi (\bar{y}^{\fixed[0.1]{\text{ }}i} \bar{z})$, and setting $a_i = \pi (\bar{x}^i \bar{z})$ gives that
\begin{equation}
\label{equ:lower_bound_on_rho(H)}
\rho ( \mathsf L (a_i)) = \frac{\max \mathsf L (a_i)}{\min \mathsf L (a_i)} \ge \frac{|\bar{x}^i \bar{z}|}{|\bar{y}^{\fixed[0.1]{\text{ }}i} \bar{z}|} = \frac{i\fixed[0.1]{\text{ }} |\bar x| + |\bar z|}{i \fixed[0.1]{\text{ }}|\bar y| + |\bar z|}.
\end{equation}
It follows that
\[
\rho (H) \ge \lim_{i \to \infty} \frac{i\fixed[0.2]{\text{ }}|\bar x| + |\bar z|}{i \fixed[0.1]{\text{ }}|\bar y| + |\bar z|} = \frac{|\bar{x}|}{|\bar{y}|} = \rho^\ast(H) \ge \rho (H) \,,
\]
whence $\rho (H) = \rho^\ast(H) = |\bar{y}|^{-1} |\bar{x}| \in \Q$. In particular, $\rho(H)$ is finite, which implies by Proposition \ref{prop:limits} that $\rho_k(H) \le k \rho(H) < \infty$ for all $k \in \mathbb N$.

We are left to show that there exists $K \in \mathbb N$ such that $\rho_{k+1}(H) - \rho_k(H) \le K$ for all $k \in \mathbb N$.
To this end, set $n = |\bar x|$ and $m = |\bar y|$. Since $\pi (\bar{x}^i \bar{z}) = \pi (\bar{y}^{\fixed[0.1]{\text{ }}i} \bar{z})$ and $\rho(H) = |\bar{y}^{\fixed[0.1]{\text{ }}i}|^{-1} |\bar{x}^i|$ for all $i \in \mathbb N$, we may assume that $|\bar{z}| \le m$, otherwise we replace $\bar x$ with ${\bar x}^{|\bar{z}|}$ and $\bar y$ with $\bar y^{|\bar{z}|}$ (we use here that $|\bar x| \ge 1$ and $|\bar y| \ge 1$). Moreover, for every $i \in \mathbb N$ we have
\[
\rho_{im+|z|} (H) \ge i n +|z|=im\rho (H) + |z|,
\]
whence
\[
(i m+|z|)\rho (H) - m (\rho (H)-1) \le (i m+|z|)\rho (H) - |z| (\rho (H)-1) = i m \rho (H) + |z| \le \rho_{i m + |z|} (H),
\]
and thus
\[
(i m+|z|)\rho (H) - \rho_{i m + |z|} (H) \le m (\rho (H)-1) \,.
\]
Now let $k \ge |z|+1$. Then there is an $i \in \N$ such that
\[
(i-1)m+|z| \le k-1 < k \le i m + |z| \,.
\]
Thus we obtain that
\[
\begin{aligned}
\rho_k (H) - \rho_{k-1} (H) & \le \rho_{i m + |z|} (H) - \rho_{(i-1)m+|z|} (H) \\
    & \le (i m + |z|)\rho (H) + m ( \rho (H)-1) - ( (i-1)m+|z|)\rho (H) \\
 & = m ( 2 \rho (H)-1) \,.
\end{aligned}
\]

\smallskip
\ref{it:3.3(2)}  Set
$$
M = \fixed[-0.65]{\text{ }}\left\{(m_1, \ldots, m_s) \in \mathbb N_0^s: \prod_{i=1}^s u_i^{m_i} \text{ has a factorization of length larger than } m_1 + \ldots + m_s\right\}\fixed[-0.25]{\text{ }} \subseteq \mathbb N_0^s\fixed[0.2]{\text{ }}.
$$
By Dickson's theorem \cite[Theorem 1.5.3]{Ge-HK06a}, we have that the set $\Min(M)$ of minimal points of $M$ (relative to the order $\preceq$) is finite.
For each $\boldsymbol m \in \Min (M)$, we thus choose an $\boldsymbol m^* = (m_1^*, \ldots, m_s^*) \in \N_0^s$ such that $|\boldsymbol m| < |\boldsymbol m^*|$ and $\prod_{i=1}^s u_i^{m_i} = \prod_{i=1}^s u_i^{m_i^*}$.
Accordingly, we set
\[
K = \max \fixed[-0.25]{\text{ }} \bigl\{ |\boldsymbol m^*| - |\boldsymbol m| \colon \boldsymbol m \in \Min (M) \bigr\}\fixed[0.15]{\text{ }}.
\]
Clearly, it is enough to show that for every $a \in H$ and every $k \in \mathsf L (a)$ with $k < \max \mathsf L (a)$ there exists $\ell \in \mathsf L (a)$ with $k < \ell \le k+K$.

Let $a = u_1^{k_1} \cdot \ldots \cdot u_s^{k_s} \in H$  with $k = k_1 + \ldots + k_s < \max \mathsf L(a)$.  Then $\boldsymbol k = (k_1, \ldots, k_s) \in M$, and hence there is an $\boldsymbol m = (m_1, \ldots, m_s) \in \Min (M)$ with $\boldsymbol m \le \boldsymbol k$. Then
$$
a  = u_1^{k_1} \cdot \ldots \cdot u_s^{k_s} =
u_1^{m_1} \cdot \ldots \cdot u_s^{m_s} \cdot u_1^{k_1 - m_1}
\cdot \ldots \cdot u_s^{k_s - m_s} =
u_1^{m_1^*} \cdot \ldots \cdot u_s^{m_s^*} \cdot u_1^{k_1 - m_1}
\cdot \ldots \cdot u_s^{k_s - m_s}
$$
has a factorization of length $|\boldsymbol m^*| - |\boldsymbol m| +  k$ with
$k < |\boldsymbol m^*| - |\boldsymbol m| +  k \le k +K$.

\smallskip
\ref{it:3.3(3)}
 Proposition \ref{3.5} implies that  $\omega(H,u) < \infty$ for all $u \in \mathcal A(H)$. It follows that $\omega(H) < \infty$, because $\mathcal A(H)$ is finite. Together with Proposition \ref{3.2}\ref{it:prop3.2(3)} and point \ref{it:3.3(1)} above, this yields that $
\mathsf t (H,u) \le \rho_{\omega(H, u)}(H) \le \rho_{\omega (H)}(H) < \infty$ for every $u \in \mathcal A(H)$, and hence $\mathsf t (H) < \infty$.
\end{proof}

\medskip
\begin{theorem} \label{3.4}
Let $H$ be an atomic monoid.
\begin{enumerate}[label={\rm (\arabic{*})}]
\item\label{it:3.4(1)} Suppose that $\mathsf c (H) < \infty$ and $H$ has accepted elasticity.
      Then the set of distances $\Delta (H)$ is finite and   there is a $K \in \mathbb N_0$ such that  $\rho_{k+1}(H) \le \rho_k(H) + K$ for all $k \in \mathbb N$.

\item\label{it:3.4(2)} If $H$ is cancellative and $\omega (H) < \infty$, then $\mathsf c (H) \le \mathsf t (H) \le \omega (H)^2$,  $\rho (H) \le \omega (H)$, and $\rho_{k+1} (H) \le  \rho_{k}(H) + \max \{1, \omega (H)-1\}$ for all $k \in \N$.
\end{enumerate}
In both cases,   $\LLc (H)$ satisfies the Structure Theorem for Unions.
\end{theorem}

\begin{proof}
If the statements on the invariants hold true, then in both cases  Theorem \ref{th:structure_theorem} shows that $\LLc (H)$ satisfies the Structure Theorem for Unions (indeed, Condition \ref{it:main(1)}\ref{it:th:structure_theorem(ii)} of Theorem \ref{th:structure_theorem} holds by the remark following Theorem \ref{th:structure_theorem}).

\ref{it:3.4(1)} We have that $\Delta (H)$ is finite by Proposition \ref{3.2}\ref{it:prop3.2(2)}, while Proposition \ref{prop:accepter_elasticity} implies the existence of a $K \in \mathbb R$ with the required properties.

\ref{it:3.4(2)} The inequalities follow from   \cite[Propositions 3.5 and 3.6]{Ge-Ka10a}.
\end{proof}

\medskip
\begin{theorem} \label{3.6}
Let $H$ be a  monoid such that $H = \mathcal F (P) \times T$, where $P \subseteq H$ is a set of pairwise non-associated primes and $T \subseteq H$ is a submonoid such that $T_{\red}$ is finitely generated.
\begin{enumerate}[label={\rm (\arabic{*})}]
\item The set of distances $\Delta (H)$ and the invariants $\rho (H)$,  $\omega (H)$, and $\mathsf t (H)$ are all finite.
\item $\LLc (H)$ satisfies the Structure Theorem for Unions.
\end{enumerate}
\end{theorem}

\begin{proof}
Proposition \ref{3.3}\ref{it:3.3(3)} implies that  $\omega (T) < \infty$.
If $\omega (T)=0$, then $T=T^{\times}$, $H$ is factorial, and the claim is trivial.
Suppose that  $\omega (T) \ge 1$. By definition of a prime, we have $\omega (H,p)=1$ for every $p \in P$.
Since $H$ is a direct product of $\mathcal F (P)$ and $T$, we infer that $\mathcal A (H) = P \cup \mathcal A (T)$ and
\[
\begin{aligned}
\omega (H)&  = \sup \{ \omega (H, u) \colon u \in \mathcal A (H)\}   = \sup \{ \omega (H, u) \colon u \in \mathcal A (T)\}   = \sup \{ \omega (T, u) \colon u \in \mathcal A (T)\} = \omega (T) \,.
\end{aligned}
\]
Furthermore, we obtain that $\Delta (H)=\Delta (T)$,   $\rho (T)= \rho (H)$, and $\rho_k (H)=\rho_k (T)$ for all $k \in \N$. Thus all assertions follow from Theorem \ref{3.4} and from Proposition \ref{3.3}.
\end{proof}

\subsection{Cancellative semigroups.} \label{subsec:3.1}

In this subsection we gather results on cancellative semigroups implying that the Structure Theorem for Unions holds. They  are simple consequences of known results. However, some of them  have never  been formulated explicitly.

\smallskip
{\bf 3.1(a) C-monoids, C-domains, and their generalizations.} For this kind of domains and monoids $H$, it has been shown that one of the following two conditions holds (\cite[Theorem 3.10]{Ga-Ge09b}):
\begin{itemize}
\item[(a)] There exists $k_0 \in \N$ such that $\rho_k (H)=\infty$ for all $k \ge k_0$.
\item[(b)] For every $k \in \N$, $\rho_k (H) < \infty$ and there is a global bound $M \in \N_0$ such that $\rho_{k+1}(H) - \rho_k (H) \le M$.
\end{itemize}
Thus $\LLc (H)$ satisfies the  Structure Theorem for Unions by Theorem \ref{th:structure_theorem}.
We provide one explicit example of $\textrm{C}$-domain and refer to \cite{Re13a, Ge-Ra-Re15c, Ka16b} for more.

Let $R$ be a Mori domain with complete integral closure $\widehat R$. Suppose that $(R  \colon \widehat R) \ne \{0\}$, that the class groups $\mathcal C_v (R)$ and $\mathcal C_v (\widehat R)$ are finite, and that the ring $S^{-1}\widehat R/ S^{-1} (R \colon \widehat R)$ is quasi-artinian, where $S$ is the monoid of regular elements of $R$. Then $R$ satisfies one of the above two conditions (see \cite[Theorems 6.2 and 7.2]{Ka16b}, and note that, by Theorem \ref{3.4},  the finiteness of the tame degree implies Condition (b)).

\smallskip
{\bf 3.1(b) Transfer Krull monoids.} This class includes all commutative Krull domains, but also maximal orders in central simple algebras. We refer to \cite{Ge16c} for a list of examples.  Let $H$ be a transfer Krull monoid. Then $H$ has a weak transfer homomorphism  $\varphi \colon H \to \mathcal B (G_0)$, where $G_0$ is a subset of an abelian group $G$ and $\mathcal B (G_0)$ is the monoid of zero-sum sequences over $G_0$.  Since $\UUc_k (H) = \UUc_k \bigl( \mathcal B (G_0) \bigr)$ for all $k \in \N$,  $\LLc (H)$ satisfies the Structure Theorem for Unions if and only if $\LLc \bigl( \mathcal B (G_0) \bigr)$ satisfies the Structure Theorem for Unions. If the Davenport constant $\mathsf D (G_0) <
 \infty$, then $\omega \bigl( \mathcal B (G_0) \bigr) < \infty$ by \cite[Theorem 3.4.10]{Ge-HK06a} and hence $\LLc \bigl( \mathcal B (G_0) \bigr)$ satisfies the Structure Theorem for Unions by Theorem \ref{3.4}
(there are several natural conditions when the finiteness of $\mathsf D (G_0)$ and of $\omega \bigl( \mathcal B (G_0) \bigr)$ are equivalent, see \cite[Theorem 4.2]{Ge-Ka10a}).
If $G_0=G$, then all sets $\mathscr{U}_k (H)$ are intervals (see \cite[Theorem 3.1.3]{Ge09a} or \cite[Theorem 4.1]{Fr-Ge08}). Recent progress on the invariants $\rho_k \bigl( \mathcal B (G) \bigr)$ can be found in \cite{Fa-Zh16a}.

\bigskip
\subsection{Semigroups of ideals.} \label{subsec:3.2}

In this subsection, all rings and domains are supposed to be commutative.
Let $H$ be a cancellative monoid and $r$ an ideal system on $H$. For an ideal $I \in \mathcal I_r (H)$ we denote by $I^{(k)}$ the $k$-fold product of $I$ with respect to $r$-multiplication. We denote by $\mathcal I_r^* (H) \subseteq \mathcal I_r (H)$ the subsemigroup of $r$-invertible $r$-ideals.
The next lemma deals with strictly $r$-noetherian ideal systems and we recall two examples. First, if $H$ is $s$-noetherian, then $H$ is strictly $s$-noetherian. Second,  let $R$ be a noetherian domain, and let $H = R\setminus \{0\}$ be the monoid of its non-zero elements. Then the $d$-system,
\[
d \colon \PPc (H) \to \PPc (H), \ \text{where} \quad X_d = \langle X \rangle_R \setminus \{0\} \quad \text{ for all } \quad X \subseteq H \,,
\]
consisting of the usual ring ideals, is strictly $d$-noetherian (\cite[Proposition 8.4]{HK98}).
We denote by $\mathcal I_d (R) = \mathcal I (R)$ the semigroup of non-zero integral ideals of $R$ (with usual ideal multiplication) and by $\mathcal I_d^* (R) = \mathcal I^* (R)$ the subsemigroup of invertible ideals.

\smallskip
\begin{lemma} \label{3.7}
Let $H$ be a cancellative monoid and let $r$ be a strictly $r$-noetherian ideal system on $H$.
\begin{enumerate}[label={\rm (\arabic{*})}]
\item\label{it:lem_3.7(1)} $\mathcal I_r (H)$ is a \BF-monoid.

\smallskip
\item\label{it:lem_3.7(2)} $\mathcal I_r^* (H)$ is a $v$-noetherian cancellative \BF-monoid. In particular, if $R$ is a noetherian domain, then the monoid $\mathcal I^* (R)$  is a $v$-noetherian cancellative \BF-monoid.
\end{enumerate}
\end{lemma}

\begin{proof}
\ref{it:lem_3.7(1)} First, we  have to show that $\mathcal I_r (H)$ is unit-cancellative. Let $I, J$ be two non-empty $r$-ideals such that $I \cdot_r J = I$. We need to prove that $J=H$. Assume to the contrary that this is not the case. Then Krull's Intersection Theorem \cite[Corollary 9.1]{HK98} implies that $\bigcap_{k \ge 1} J^{(k)} = \emptyset$. But $I = I \cdot_r J$ implies that $
I = I \cdot_r J = I \cdot_r J^{(2)} = \ldots = I \cdot_r J^{(k)}$ for all $k \in \N$, and hence
\[
I \subseteq \bigcap_{k \ge 1} J^{(k)} = \emptyset \,,
\]
a contradiction. Thus $\mathcal I_r (H)$ is a monoid, and it is a \BF-monoid by Proposition \ref{3.5}\ref{it:prop_3.5(2)}.

\smallskip
\ref{it:lem_3.7(2)} Clearly, every $r$-invertible $r$-ideal is $r$-cancellative (\cite[Section 13.1]{Ge-HK06a}), and hence $\mathcal I_r^* (H)$ is a cancellative monoid. The remaining statements follow from \cite[Example 2.1]{Ge-Ha08a}.
\end{proof}

Lemma \ref{3.7} shows that semigroups of ($r$-invertible) $r$-ideals are \BF-monoids which often satisfy additional noetherian properties. The verification of the assumptions needed to enforce the validity of the Structure Theorem for Unions clearly depends on the ideal system $r$ and the class of monoids and domains under consideration. Much is already known for $r$-invertible $r$-ideals which are cancellative. For example, if $R$ is a $v$-noetherian weakly Krull domain with non-zero conductor, then $\LLc \bigl( \mathcal I_v^* (R) \bigr)$  satisfies the Structure Theorem for Unions (use \cite[Theorem 3.7.1]{Ge-HK06a} and Subsection \ref{subsec:3.1}).
To provide an explicit example of not necessarily cancellative semigroups of ideals, we consider one-dimensional noetherian domains. Note that by an irreducible ideal we mean an ideal which is an irreducible element in the  semigroup of ideals under consideration.

\smallskip
\begin{corollary} \label{3.8}~
Let $R$ be a one-dimensional noetherian domain such that its integral closure $\overline R$ is a finitely generated $R$-module, and let $\mathfrak f = (R \DP \overline R)$ and $\mathcal P^*  = \{ \mathfrak p \in \mathfrak X (R) \colon \mathfrak p \supseteq \mathfrak f \}$. If the following condition:
\begin{enumerate}[label={\rm ({\bf C})}]
\item\label{it:condition(C)} For every $\mathfrak p \in \mathcal P^*$, there are only finitely many irreducible primary ideals $\mathfrak q$ with $\sqrt{\mathfrak q} = \mathfrak p$,
\end{enumerate}
holds, then $\LLc \bigl( \mathcal I (R) \bigr)$ satisfies the Structure Theorem for Unions. In particular:
\begin{enumerate}[label={\rm (\arabic{*})}]
\item\label{it:lem_3.8(1)} If $R$ is a Cohen-Kaplansky domain, then Condition \ref{it:condition(C)} holds.

\item\label{it:lem_3.8(2)} If $R$ is an order in a quadratic number field $K$ with discriminant $d_K$ and conductor $\mathfrak f = f \Z$, then Condition \ref{it:condition(C)} is equivalent to each of the following{\rm \,:}
    \begin{enumerate}[label={\rm ({\bf C}\arabic{*})}]
    \item\label{it:condition(C1)} For every $\mathfrak p \in \mathcal P^*$, there are only finitely many invertible irreducible primary ideals $\mathfrak q$ with $\sqrt{\mathfrak q} = \mathfrak p$.

    \item\label{it:condition(C2)} For every $\mathfrak p \in \mathcal P^*$, there  is precisely one prime ideal of $\overline R$ lying above $\mathfrak p$.

    \item\label{it:condition(C3)} For every prime divisor $p$ of $f$ the Legendre symbol $\Big( \frac{d_K}{p} \Bigr) \ne -1$.
    \end{enumerate}
\end{enumerate}
\end{corollary}

\begin{proof}
Let $\pi \colon \mathfrak X (\overline R) \to \mathfrak X (R)$ be defined by $\pi (\mathfrak P) = \mathfrak P \cap R$.
    By the Primary Decomposition Theorem, every ideal $I \in \mathcal I (R)$ is, up to order, a unique product of primary ideals having pairwise distinct radicals.  The set $\mathcal P = \mathfrak X (R) \setminus \mathcal P^*$ is the set of prime elements of $\mathcal I (R)$, and every ideal $I \in \mathcal I (R)$ can be written uniquely as a product of powers of prime elements and of primary elements with radicals in $\mathcal P^*$.
    Therefore $\mathcal I (R) = \mathcal F (\mathcal P) \times T$ where $T \subseteq \mathcal I (R)$ is the submonoid generated by atoms which are not prime.   Condition \ref{it:condition(C)} states that the monoid $T$ is finitely generated. Thus the assumptions of Theorem  \ref{3.6} hold, and $\LLc \bigl( \mathcal I (R) \bigr)$ satisfies the Structure Theorem for Unions.

It remains to discuss the two special cases.

\smallskip
\ref{it:lem_3.8(1)} By definition, a domain is  a Cohen-Kaplansky domain if it is atomic and has only finitely many non-associated atoms. Suppose $R$ is a Cohen-Kaplansky domain. Then $R$ is a one-dimensional, semilocal noetherian domain, $\overline R$ is a finitely generated $R$-module, and both ideal semigroups $\mathcal I (R)$ and $\mathcal I^* (R)$ are finitely generated by \cite[Theorem 4.3]{An-Mo92}. Thus Condition \ref{it:condition(C)} holds. Note that Cohen-Kaplansky domains are precisely the domains for which $\mathcal I (R)$ is finitely generated.

\smallskip
\ref{it:lem_3.8(2)} Suppose that $R$ is an order in a quadratic number field. The equivalence of the given conditions is classical and can be found in \cite[p. 36]{Ge-Le90} and \cite[Theorem 5.8.8]{HK13a}. Clearly, \ref{it:condition(C2)} states that   $\pi \colon \mathfrak X (\overline R) \to \mathfrak X (R)$ is bijective.
\end{proof}

\bigskip
\subsection{Semigroups of modules.} \label{subsec:3.3}
The study of direct-sum decomposition of  modules is a classical topic of module theory. There is an overwhelming literature investigating for which classes of modules   the Krull-Remak-Schmidt-Azumaya Theorem (stating that direct-sum decomposition is unique) holds, or how badly it fails.

The last decade has seen  a new  semigroup-theoretical approach to the study of direct-sum decomposition. To highlight this approach, we need some notation.
Let $R$ be a (not necessarily commutative) ring and $\mathcal C$ a class of  (left) $R$-modules which is closed under finite direct sums, direct summands, and isomorphisms. This means that whenever $M,M_1$ and $M_2$ are $R$-modules with $M \cong M_1 \oplus M_2$, we have $M \in \mathcal C$ if and only if $M_1, M_2 \in \mathcal C$. Let $\mathcal V (\mathcal C)$ denote a set of representatives of isomorphism classes in $\mathcal C$, and for a module $M$ in $\mathcal C$ let $[M]$ denote the unique element in $\mathcal V (\mathcal C)$ isomorphic to $M$.
Then  $\mathcal V (\mathcal C)$ is an (additive) reduced commutative semigroup, where the operation is defined by $[M]+[N] = [M\oplus N]$, the identity is the zero-module, and the irreducible elements are the isomorphism classes of indecomposable modules. Of course, the semigroup $\mathcal V (\mathcal C)$ carries all the information about the direct-sum behavior of modules in $\mathcal C$.

If $\End_R (M)$ is semilocal for each $M \in \mathcal C$, then $\mathcal V (\mathcal C)$ is a Krull monoid by a result of Facchini \cite[Theorem 3.4]{Fa02}. In this case $\LLc \bigl( \mathcal V (\mathcal C) \bigr)$ can be studied with the theory of Krull monoids and we refer back to Subsection \ref{subsec:3.1}. This pioneering result was the starting point pushing a new strategy to understand the direct-sum behavior of modules by studying the algebraic and arithmetic properties of $\mathcal V(\mathcal C)$ (see \cite{Wi-Wi09a, Fa12a, Ba-Wi13a} for surveys and \cite[p. 315]{Ba-Ge14b} for a detailed program description; in \cite{Di07a} $\omega$-invariants of semigroup of modules are studied in terms of a so-called semi-exchange property).

To begin, direct-sum decompositions can be as wild as factorizations in arbitrary commutative semigroups. Indeed, we know from a result of Facchini and Wiegand \cite[Theorem 2.1]{Fa-Wi04} that every reduced Krull monoid $H$ is isomorphic to a monoid of modules $\mathcal V (\mathcal C)$, where $\mathcal C$ is a class of finitely generated projective modules with semilocal endomorphism rings.
And by results of Bergman and Dicks, for every reduced commutative semigroup $H$ (with order-unit) there exists a class $\mathcal C$ of finitely generated projective  modules over a right and left hereditary $k$-algebra $R$ such that $\mathcal V (\mathcal C)$ is isomorphic to $H$ (\cite[Corollary 5]{Fa06b}).

So far, the focus of research has been on classes $\mathcal C$ of modules for which $\mathcal V (\mathcal C)$ is  cancellative, but other classes of semigroups appear naturally (see   the Seven-Step hierarchy given by Facchini in \cite{Fa06a}).
In order to apply Theorem \ref{3.6} to a monoid of modules $\mathcal V ( \mathcal C)$, we analyze its assumptions in the present setting. Clearly, $\mathcal V (\mathcal C)$ being finitely generated means that, up to isomorphism, there are only finitely many indecomposable $R$-modules, and  unit-cancellativity reads as follows:
\[
\text{If} \ M, N \ \text{are modules from} \ \mathcal C  \ \text{such that} \quad M \cong M \oplus N, \quad \text{then} \quad N=0 \,.
\]
Thus unit-cancellativity states that all modules have to be directly finite (or {\it Dedekind finite}). Clearly, finitely generated modules over commutative rings are directly finite. Furthermore, modules with finite uniform dimension are directly finite, and so are finitely generated projective modules over unit-regular rings. For more information on directly finite modules we refer to \cite{Go79a, Ha-Gu-Ki04a, Br-Cu-Sc11a}.

We summarize our discussion in the following corollary.

\medskip
\begin{corollary} \label{3.9}
Let $R$ be a  ring and $\mathcal C$ be a class of  $R$-modules which is closed under finite direct sums, direct summands, and isomorphisms. Suppose that all modules in $\mathcal C$ are directly finite and that there are, up to isomorphism, only finitely many indecomposable modules in $\mathcal C$. Then $\LLc \bigl( \mathcal V (\mathcal C) \bigr)$ satisfies the Structure Theorem for Unions.
\end{corollary}

\begin{proof}
This follows immediately from Theorem \ref{3.6}.
\end{proof}

\bigskip
\subsection{Power monoids.} \label{subsec:3.4}

Given an additively written commutative semigroup $H$, the power semigroup $\mathscr{P}_\fin(H)$ is the semigroup of all finite non-empty subsets of $H$ endowed with the operation of set addition.
The study of sumsets is in the center of interests in arithmetic combinatorics \cite{Gr13a}, but the abstract semigroup theoretical point of view was adopted only recently in   \cite{Fa-Tr17a}.

In general, $\mathscr{P}_\fin(H)$ need not be unit-cancellative. If, for instance, $H$ contains a finite non-trivial subgroup $G$, then $G \in \mathscr{P}_\fin(H)$ and $G + G = G$, but $G$ is not a unit in $\mathscr{P}_\fin(H)$. This  shows that a necessary condition for $\mathscr{P}_\fin(H)$ to be a monoid is that $H$ is torsion-free, and for a commutative, cancellative semigroup it is well known that torsion-freeness is equivalent to being linearly orderable. We recall that a semigroup $H$ is \textit{linearly orderable} if there exists a total order $\preceq$ on the set $H$ such that $x+z \prec y+z$ for all $x, y, z \in H$ with $x \prec y$ (note that linearly orderable semigroups are cancellative).

In general, we cannot expect arithmetical finiteness properties from $\mathscr{P}_\fin(H)$ (to provide an example, if $H = (\N_0,+)$, then the set of distances of $\mathscr{P}_\fin(H)$ is equal to $\N$ by \cite{Fa-Tr17a}). However, finitely generated subsemigroups of $\mathscr{P}_\fin(H)$ do satisfy arithmetical finiteness properties. To see this,
consider a subset $\mathcal{E} \subseteq \mathscr{P}_\fin(H)$. We denote by $H(\mathcal{E})$ the subsemigroup of $\mathscr{P}_\fin(H)$ generated by $\mathcal{E}$. Clearly, if $\mathcal{E}$ is a family of singletons, then $H(\mathcal{E})$ is isomorphic to the subsemigroup of $H$ generated by the elements occurring in the singletons of $\mathcal E$. Moreover, we have the following result.
\medskip
\begin{corollary} \label{3.10}
Let $H$ be a reduced linearly orderable \BF-monoid, and let $\mathcal{E} \subseteq \mathcal P_\fin(H)$ be a finite  subset.   Then $H (\mathcal E)$ is a finitely generated \BF-monoid and  $\LLc \bigl(H(\mathcal{E}) \bigr)$ satisfies the Structure Theorem for Unions.
\end{corollary}
\begin{proof}
It follows from  \cite{Fa-Tr17a} that $H (\mathcal E)$ is a finitely generated \BF-monoid, and hence Theorem \ref{3.6} implies the validity of the Structure Theorem for Unions.
\end{proof}

\smallskip
The special case of Corollary \ref{3.10}, when $H$ is the monoid of non-negative integers under addition, lies at the heart of the following remark, in which
we highlight a few essential differences between the classical cancellative setting and the broader setting considered in the present work.

\smallskip

\begin{remarks} \label{important}
 Let $H$ be an atomic monoid.

(1)  By \cite[Theorem 3.6]{HK98}, every reduced finitely generated commutative semigroup is $s$-noetherian. However, in contrast to Proposition \ref{3.1} it need not be atomic (see \cite[Theorem 3 and Example 6]{Ro-GS-GG04b}).

\smallskip
(2) If  $H$ is cancellative and $a \in H$ with $|\mathsf Z (a)|\ge 2$, then $2 + \sup \Delta ( \mathsf L (a)) \le \mathsf c (a)$ by \cite[Lemma 1.6.2]{Ge-HK06a}, but in the more general setting of Proposition \ref{3.2}, the inequality $1 + \sup \Delta ( \mathsf L (a)) \le \mathsf c (a)$ can be sharp.

\smallskip
(3) Let $\sim_H = \{ (x,y) \in \mathsf Z (H) \times \mathsf Z (H) \colon \pi (x) = \pi (y) \}$ be the monoid of relations of $H$, and assume first that $H$ is cancellative. Then $\sim_H$ is a fundamental tool for studying the arithmetic of $H$. In particular, it is easy to see that the canonical embedding $j: {\sim_H} \hookrightarrow \mathsf Z (H) \times \mathsf Z (H)$ is a divisor homomorphism. Thus $\sim_H$ is a Krull monoid, and it is finitely generated if so is $H$.

However, if $H$ is not cancellative, things are quite different.
Indeed, suppose that $H$ is reduced and there exist $a,b,c \in H$ with $b \ne c$ and $ab=ac$. Then none of $a$, $b$, and $c$ is a unit,
and hence none of $\mathsf Z(a)$, $\mathsf Z(b)$, and $\mathsf Z(c)$ is empty.
Let $z \in \mathsf Z(ab)$, $z_a \in \mathsf Z(a)$, $z_b \in \mathsf Z(b)$, and $z_c \in \mathsf Z(c)$. Then $(z,z)$ and $(z_a, z_a)$ both belong to $\sim_H$, and $(z,z) = (z_a, z_a)\fixed[0.3]{\text{ }}  (z_b, z_c)$ in $\mathsf Z(H) \times \mathsf Z(H)$. Yet, $(z_b, z_c)$ is not in $\sim_H$, and it follows that
$j$ is not a divisor homomorphism.

Now, for a fixed integer $n \ge 2$ let $H$ be the submonoid of $\PPc_\fin(\mathbb N_0)$ generated by the sets $\llb 0, 1 \rrb $ and $A = \{1\} \cup \bigl(2 \cdot \llb 0, n \rrb \bigr)$. It is seen that $h \fixed[0.1]{\text{ }} \llb 0, 1 \rrb + k \fixed[0.1]{\text{ }} A$ is an interval for all $h \in \mathbb N$ and $k \in \mathbb N_0$. Therefore, we find that $\llb 0, 1 \rrb$ is prime in $H$ and $\mathcal A (H) = \{\llb 0,1 \rrb, A\}$. In particular, $H$ is a finitely generated, reduced, atomic monoid. However, we claim that $\sim_H$ is not finitely generated.

In fact, we will show that $\bigl(\llb 0, 1 \rrb + k A, (2nk+1) \llb 0, 1 \rrb \bigr)$ is an atom of $\sim_H$ for every $k \in \mathbb N_0$. (Since $H$ is written additively, so will be its factorization monoid $\mathsf Z (H)$.)
Suppose for a contradiction that this is not true. Then there must exist $k \in \mathbb N_0$, $h \in \llb 0, k \rrb$, and $\ell \in \llb 0, 2nk+1 \rrb$ with $h+\ell \ge 1$ such that
\[
\bigl(\llb 0, 1 \rrb + k A, (2nk+1) \llb 0, 1 \rrb\bigr) = \bigl(h A, \ell \llb 0, 1 \rrb\bigr) \bigl(\llb 0, 1 \rrb + (k-h)  A, ({2nk+1 - \ell}) \llb 0, 1 \rrb \bigr)
\quad\text{in}\quad
\sim_H.
\]
But this implies that $hA = \llb 0, \ell \rrb$, which is impossible, since $hA$ is not an interval for $h \ne 0$, and on the other hand, $\{0\} = 0A = \llb 0, \ell \rrb$ only for $\ell = 0$ (recall that we are requiring $h+\ell \ge 1$).

\smallskip
(4) Let $P$ be a maximal set of pairwise non-associated prime elements in $H$, and let $T$ be the set of all $a \in H$ for which there does not exist any $p \in P$ such that $p \mid a$. Of course, $T$ is a submonoid of $H$, and by Lemma \ref{3.1} we have $H = \mathcal F (P) \fixed[0.2]{\text{ }} T$. If $H$ is cancellative, then $H$ is isomorphic to the direct product $\mathcal F(P) \times T$. However, this need not be the case if $H$ is not cancellative.

Indeed, assume as in point (3) that $H$ is the submonoid of $\mathscr{P}_\fin(\mathbb N_0)$ generated by $\llb 0, 1 \rrb$ and $A = \{1\} \cup \bigl(2 \cdot \llb 0, n \rrb \bigr)$, where $n$ is a fixed integer $\ge 2$. We have already noted that $H$ is reduced, and $\llb 0, 1 \rrb$ is prime and $A$ is irreducible in $H$. So we necessarily have $P = \fixed[-0.25]{\text{ }}\bigl\{\llb 0, 1 \rrb \bigr\}$ and $T = \{k A: k \in \mathbb N_0\}$, and $H$ is not isomorphic to $\mathcal F(P) \times T$, because $\llb 0, 1 \rrb + A = (2n+1) \fixed[0.1]{\text{ }} \llb 0, 1 \rrb$ in $H$.

\smallskip
(5) It is known that if $H$ is finitely generated and cancellative, then $H$ has accepted elasticity (\cite[Theorem 3.1.4]{Ge-HK06a}). Yet, this need not be true in the non-cancellative setting.

To see why, let $H$ be given as in the second paragraph of point (4), and let $X$ be a non-unit element of $H$. Then $X = h \fixed[0.2]{\text{ }} \llb 0, 1 \rrb + l A$ for some $h, l \in \mathbb N_0$ with $h+l \ge 1$. We seek an explicit description of the set of lengths, $\mathsf L(X)$, of $X$. To start with, note that if $h = 0$, then $\mathsf L(X) = \{l\}$, since $lA \ne jA$ for all $j \in \mathbb N_0$ with $j \ne l$ and, on the other hand, $\llb 0, 1 \rrb$ cannot be a divisor of $lA$ (in $H$), otherwise $lA$ would be an interval. So assume $h \ge 1$, and let $h = 2nq + r$, where $q \in \mathbb N_0$, $r \in \llb 0, 2n-1 \rrb$, and $q+r \ge 1$. Then
$X = \llb 0, 2n(q+l)+r \rrb$, and we look for all pairs $(j,m)$ of non-negative integers such that
$$
\llb 0, 2n(q+l)+r \rrb = j \fixed[0.25]{\text{ }} \llb 0, 1 \rrb + mA.
$$
This condition implies that $2n(q+l)+r = j + m \max A = j + 2mn$, and is actually possible if and only if
$$
j = 2n(q+l-m)+r
\quad\text{for some}\quad
m \in \llb 0, q+l-\varepsilon_r \rrb,
$$
where $\varepsilon_r = 1$ if $r = 0$ and $\varepsilon_r = 0$ otherwise. It follows that
\begin{equation}
\label{equ:closed_form_for_L(X)}
\mathsf L(X) =
\fixed[-0.2]{\text{ }}\bigl\{(2n-1)x+q+l+r: x \in \llb \varepsilon_r, q+l \rrb \bigr\}\fixed[0.1]{\text{ }},
\end{equation}
whence $\min \mathsf L(X) =  q+l + r + (2n-1)\fixed[0.1]{\text{ }}  \varepsilon_r$, $\max \mathsf L(X) = 2n(q+l) + r$, and
$$
\rho(X) = \frac{2n(q+l) + r}{q+l + r + (2n-1)\fixed[0.1]{\text{ }} \varepsilon_r} < 2n = \lim_{q+l \to \infty} \frac{2n(q+l) + r}{q+l + r + (2n-1)\fixed[0.1]{\text{ }}  \varepsilon_r},
$$
where the limit on the right-most side is uniform with respect to the parameter $r$.

Therefore, we conclude that $\rho(H) = 2n \ne \rho(X)$ for every $X \in H$. In particular, $H$ does not have accepted elasticity, although it is finitely generated (in fact, $2$-generated).

With this in hand, we observe that, by Proposition \ref{3.3}\ref{it:3.3(1)}, $\rho_k(H) < \infty$ for all $k \in \mathbb N$. We want to show that $\rho_{k+1}(H) - \rho_k(H) \le 2n$ for all sufficiently large $k \in \mathbb N$.

Indeed, let $k$ be an integer $\ge 2n$, and write $k = (2n-1)t+s$, where $t \in \mathbb N$ and $s \in \llb 0, 2n-2 \rrb$.
Our goal is to compute $\rho_k$. To this end, we seek all triples $(q, l, r) \in \mathbb N_0^3$ with $q+r \ge 1$ and $0 \le r < 2n$ such that $k \in \mathsf L\bigl(\llb 0, 2n(q+l)+r\rrb\bigr)$, which is equivalent by \eqref{equ:closed_form_for_L(X)} to having that
\begin{equation}
\label{equ:condition_for_rhok}
k = (2n-1)x + q + l + r
\quad\text{for some}\quad
x \in \llb \varepsilon_r, q+l \rrb.
\end{equation}
Since we have already seen that $\max \mathsf L\bigl(\llb 0, 2n(q+l)+r\rrb\bigr) = 2n(q+l)+r$, we can just look for the set $\mathcal M$ of those triples $(q, l, r) \in \mathbb N_0^3$ with $q+r \ge 1$ and $0 \le r < 2n$ for which Condition \eqref{equ:condition_for_rhok} is fulfilled and $2n(q+l)+r$ is maximal (note that $\mathcal M$ is non-empty).

Let $(q,l,r) \in \mathcal M$. Then \eqref{equ:condition_for_rhok} is necessarily satisfied with $x = \varepsilon_r$. Otherwise, taking $x^\prime = x - 1$ and $q^\prime = q + 2n-1$ would give that $(q^\prime, l, r) \in \mathbb N_0^3$, $q^\prime + r \ge 1$, and $
x^\prime \in \llb \varepsilon, q^\prime + l \rrb$, but also
$$
2n(q^\prime + l) + r > 2n(q+l) + r
\quad\text{and}\quad
k = (2n-1)x^\prime + q^\prime + l + r,
$$
in contradiction to the maximality of the triple $(q,l,r)$. This shows that $\mathcal M$ is the set of all triples $(q,l,r) \in \mathbb N_0^3$ with $q+r \ge 1$ and $0 \le r < n$ for which $2n(q+l)+r$ is maximal and
\begin{equation}
\label{equ:final_condition_for_rhok}
q + l = (2n-1)(t-\varepsilon_r) + s - r.
\end{equation}
Now we notice that $(2n-1)(t-\varepsilon_r) + s - r \ge 1$ for every $r \in \llb 0, 2n-1 \rrb$, since having assumed $k \ge 2n$ implies that $t \ge 1$, and that $s \ge 1$ for $t = 1$. This means that, however we choose an $r \in \llb 0, 2n - 1 \rrb$, there are $q, l \in \mathbb N_0$ with $q+r \ge 1$ such that \eqref{equ:final_condition_for_rhok} is satisfied, which ultimately yields
$$
\rho_k(H) = \max\bigl\{(2n-1)(t-\varepsilon_r) + s - r: r \in \llb 0, n-1 \rrb \bigr\} = 2n(k-1) + 1.
$$
Putting it all together, we thus conclude that $\rho_{k+1}(H) - \rho_k(H) = 2n$ for every integer $k \ge 2n$.

Incidentally, it seems worth noting that $\omega(H) = 2n+1$: This comes a consequence of what we have already observed in points (3) and (4), namely the fact that $\llb 0,1 \rrb$ is prime and $A$ divides $k \fixed[0.25]{\text{ }} \llb 0,1 \rrb$ for $k \ge 2n+1$. In particular, we obtain from here and the above that $\rho_{k+1}(H) - \rho_k(H) \le \max(1, \omega(H) - 1)$ for all sufficiently large $k \in \mathbb N$, which is consistent with the bound provided in the cancellative setting by Theorem \ref{3.4} and leads us to ask whether the same inequality is actually true for any finitely generated monoid (if not for every monoid).

\smallskip
(6) If $H$ is as in Theorem \ref{3.4}, then the assumptions made in Theorem \ref{th:structure_theorem}\ref{it:main(2)} need not hold, not even in the cancellative case. To be more precise, let $H$ be a cancellative finitely generated monoid. Then $\rho (H) < \infty$ and there is an $L^* \in \mathscr{L} (H)$ such that $\rho (L^*) = \rho (H)$ (see \cite[Theorem 3.1.4]{Ge-HK06a}). However, there need not be an $L^* \in \mathscr{L} (H)$ which is an AP with difference $\min \Delta (H)$ and satisfies $\rho (L^*) = \rho (H)$ (see \cite[Example 3.1]{Fr-Ge08}). For numerical monoids, sufficient conditions implying the existence of such an $L^* \in \mathscr{L} (H)$ are given in \cite[Theorem 6.6]{Bl-Ga-Ge11a}.
\end{remarks}

\medskip
\section{A  semigroup not satisfying the Structure Theorem for Unions} \label{4}
\medskip

In this section we provide the first example of a semigroup  whose system of sets of lengths does not satisfy the Structure Theorem for Unions. Our example is, actually, a  locally tame Krull monoid with finite set of distances. Recall that  globally tame cancellative monoids satisfy  the Structure Theorem for Unions (Theorem \ref{3.4}), as well as the Structure Theorem for Sets of Lengths  (\cite[Theorem 5.1]{Ge-Ka10a}).

\begin{proposition} \label{4.1}
Let $L \subseteq \N_{\ge 2}$ be a finite subset such that $|L| \geq 2$. Then there exists a finitely generated reduced Krull monoid $H$ such that $L \in \LLc (H)$ and the following properties hold{\rm \,:}
\begin{enumerate}[label={\rm (\roman{*})}]
\item\label{4.1(i)} $\max\Delta (H)  = \max\Delta (L)$.
\item\label{4.1(ii)} $\UUc_{\min L} (H) = L$.
\item\label{4.1(iii)} $\rho_{\ell \min L + \nu} (H) = \ell \max L + \nu$ for all $\nu \in \llb 0, \min L - 1 \rrb$ and $\ell \in \N_0$.
\end{enumerate}
\end{proposition}

\smallskip
\noindent
\textit{Comment.} Since $L \in \LLc (H)$, it is clear that $\Delta (L) \subseteq \Delta (H)$. Yet, we cannot expect equality in general, as we know from Proposition \ref{prop:delta_sets} that
$\min \Delta (H) = \gcd \Delta (H)$, which does not necessarily hold for $\Delta (L)$.

\smallskip
\begin{proof}
Set $L = \{m_1, m_2+1, \ldots, m_s+1\} \subseteq \N_{\ge 2}$, where $s=|L|\geq 2$, $m_1 \le m_2$, and $m_i < m_{i+1}$ for every $i \in \llb 2, s-1 \rrb$.
Take $
P = \bigl\{u_{i,j}: i \in \llb 1, s \rrb \text{ and } j \in \llb 1, m_i \rrb \bigr\}
$
to
be a finite set with $|P|= m_1 + \ldots + m_s$, and let $F = \mathcal F (P)$ be the free abelian monoid with basis $P$. For each $i \in \llb 2, s \rrb$ we set
\begin{equation}
\label{equ:definition-of-ui0}
u_{i,0} = u_{1,1} \cdot \ldots \cdot u_{1, m_1} ( u_{i,1} \cdot \ldots \cdot u_{i, m_i} )^{-1} \in \mathsf q (F) \fixed[0.35]{\text{ }},
\end{equation}
and we define $H=[\mathcal{A}]\subseteq \mathsf{q}(F)$, where
\[
\mathcal{A} = \fixed[-0.25]{\text{ }}\bigl\{u_{i,j}: i\in \llb 1,s \rrb \text{ and } j\in \llb 1,m_i \rrb \bigr\}\cup \bigl\{u_{i,0}: i\in \llb 2,s \rrb\bigr\}\fixed[0.25]{\text{ }}.
\]
We note that $\mathsf{q}(H) = \mathsf{q}(F)$, and we proceed in a series of steps.

\smallskip
\noindent
{\bf 1. Description of $H$}. To start with, we seek a more explicit description of the monoid $H$.
For this, note first that, if $\alpha_{i,j}\in\Z$ for $i\in \llb 1,s \rrb$, $j\in \llb 1,m_i \rrb$ and $\beta_i\in\Z$ for $i\in \llb 2,s \rrb$, then we have by \eqref{equ:definition-of-ui0} that
\begin{equation}\label{E:Basis}
\prod_{i=1}^s\prod_{j=1}^{m_i}u_{i,j}^{\alpha_{i,j}}\prod_{i=2}^su_{i,0}^{\beta_i}=
\prod_{j=1}^{m_1}u_{1,j}^{\alpha_{1,j}+\sum_{i=2}^s\beta_i}\prod_{i=2}^s\prod_{j=1}^{m_i}u_{i,j}^{\alpha_{i,j}-\beta_i}.
\end{equation}
Fix $x\in\mathsf{q}(F)$. Then there are uniquely determined $\gamma_{i,j}(x)\in\Z$, for $i\in \llb 1,s \rrb$ and $j\in \llb 1,m_i \rrb$, such that
\begin{equation}
\label{equ:generic-element-of-q(F)}
x=\prod_{i=1}^s\prod_{j=1}^{m_i}u_{i,j}^{\gamma_{i,j}(x)}\fixed[0.25]{\text{ }}.
\end{equation}
Now, for each $i\in \llb 1,s \rrb$ we define
\[
\gamma_i(x)=\min\{ \gamma_{i,1}(x),\ldots,  \gamma_{i,m_i}(x)\}
\quad\text{and}\quad
\gamma'_i(x)=\max\{0,-\gamma_i(x)\}
\fixed[0.35]{\text{ }}.
\]
Furthermore, we denote by $\mathsf{Z}_{\mathcal{A}}(x)$ the set of all tuples
$
\bigl((\alpha_{i,j})_{1 \le i \le s, 1 \le j \le m_i}, \beta_2,\ldots,  \beta_s\bigr)\in\N_0^{m_1 +\ldots + m_s+s-1}
$ such that
\[
x=\prod_{i=1}^s\prod_{j=1}^{m_i}u_{i,j}^{\alpha_{i,j}}\prod_{i=2}^su_{i,0}^{\beta_i}\fixed[0.35]{\text{ }}.
\]
Lastly, for every $\alpha= \fixed[-0.25]{\text{ }}\bigl((\alpha_{i,j})_{1 \le i \le s, 1 \le j \le m_i}, \beta_2,\ldots,  \beta_s\bigr) \in \N_0^{m_1 +\ldots + m_s+s-1}$ we let
\[
|\alpha|=\sum_{i=1}^s\sum_{j=1}^{m_i}\alpha_{i,j}+\sum_{i=2}^s\beta_i
\]
be the length of $\alpha$, and we set
\[
\mathsf{L}_{\mathcal{A}}(x)=\fixed[-0.25]{\text{ }} \bigl\{|\alpha| \colon \alpha\in \mathsf{Z}_{\mathcal{A}}(x)\bigr\}\fixed[0.25]{\text{ }}.
\]
By definition, $x\in H$ if and only if $\mathsf{Z}_{\mathcal{A}}(x)\neq \emptyset$.
Our goal is to give a simpler description of $\mathsf{Z}_{\mathcal{A}}(x)$. To this end, let $B(x)$ be the set of all tuples $(\beta_2,\ldots,  \beta_s)\in\N_0^{s-1}$ such that
\[
\sum_{i=2}^s\beta_i\leq \gamma_1(x)
\quad\text{and}\quad
\beta_i\geq -\gamma_i(x) \,\ \text{for each} \,\  i \in \llb 2,s \rrb \fixed[0.35]{\text{ }}.
\]
It is clear that $B(x)\neq \emptyset$ if and only if $
\gamma_1(x)\geq \sum_{i=2}^s \gamma^\prime(x)$, and we will prove that $B(x) \ne \emptyset$ if and only if $\mathsf{Z}_{\mathcal{A}}(x)\neq \emptyset$, whence we can conclude that
\begin{equation}\label{E:H}
H=\left\{x\in \mathsf{q}(F) \colon \gamma_1(x)\geq\sum_{i=2}^s\gamma'_i(x)\right\}.
\end{equation}
To begin, let $\beta=(\beta_2,\ldots,  \beta_s)\in B(x)$, and for every $j\in \llb 1,m_1 \rrb$ define
$\alpha_{1,j}=\gamma_{1,j}(x)-\sum_{i=2}^s \beta_i$. Then
$$
\alpha_{1,j}\geq \gamma_1(x)-\sum_{i=2}^s \beta_i\geq 0 \fixed[0.35]{\text{ }}.
$$
Similarly, for all $i\in \llb 2,s \rrb$ and $j\in \llb 1,m_i \rrb$ set $\alpha_{i,j}=\gamma_{i,j}(x)+\beta_i$. Again, we have $\alpha_{i,j}\geq \gamma_{i}(x)+\beta_i\geq 0$. It follows by \eqref{E:Basis} and  \eqref{equ:generic-element-of-q(F)} that
\[
\mathsf{Z}_x(\beta):=\fixed[-0.25]{\text{ }}\bigl((\alpha_{i,j})_{1 \le i \le s, 1 \le j \le m_i},\beta\bigr) \fixed[-0.25]{\text{ }}\in \mathsf{Z}_{\mathcal{A}}(x).
\]
Consider the map $\mathsf{Z}_x\colon B(x)\rightarrow \mathsf{Z}_{\mathcal{A}}(x)$, $\beta\mapsto \mathsf{Z}_x (\beta)$. This is obviously injective, and we want to show that it is also surjective. Indeed, let $\alpha= \fixed[-0.25]{\text{ }}\bigr((\alpha_{i,j})_{1 \le i \le s, 1 \le j \le m_i},\beta_2,\ldots,  \beta_s\bigl)\fixed[-0.25]{\text{ }}\in\mathsf{Z}_{\mathcal{A}}(x)$. By \eqref{E:Basis} and \eqref{equ:generic-element-of-q(F)}, we have
$$
\alpha_{1,j}=\gamma_{1,j}(x)-\sum_{i=2}^s\beta_i \,\ \text{for each} \,\  j\in \llb 1, m_1 \rrb
$$
and
$$
\alpha_{i,j}=\gamma_{i,j}(x)+\beta_i \,\ \text{for all} \,\  i\in \llb 2,s \rrb \,\ \text{and} \,\ j\in \llb 1,m_i\rrb\fixed[0.35]{\text{ }}.
$$
So, since $\alpha_{i,j}\geq 0$ for every $i\in \llb 1,s \rrb$ and $j\in \llb 1,m_i \rrb$, we get $\beta=(\beta_2,\ldots,  \beta_s)\in B(x)$, and this entails (by construction) that $\mathsf{Z}_x(\beta)=\alpha$.
Therefore, we can conclude that the map
$\mathsf{Z}_x\colon B(x)\rightarrow \mathsf{Z}_{\mathcal{A}}(x)$ is bijective, which, in turn, implies that $B(x)$ is non-empty if and only if so is $\mathsf Z_\mathcal{A}(x)$, as was desired. In addition,
\begin{equation}
\label{equ:rewriting-LA(x)}
{\sf L}_\mathcal{A}(x) = \fixed[-0.25]{\text{ }} \bigl\{\left|\mathsf{Z}_x(\beta)\right|: \beta \in B(x)\bigr\}.
\end{equation}

\smallskip
\noindent
{\bf 2. Claim:}  $H$ is a reduced, finitely generated Krull monoid.

Since $\gamma_i(1)=0$ for all $i\in \llb 1,s \rrb$, we find $B(1)=\{(0,\ldots, 0)\}$. Thus, we also have $\mathsf{Z}_{\mathcal{A}}(1)=\{(0,\ldots,  0)\}$. Now let $x\in H^\times$, and choose $z\in\mathsf{Z}_{\mathcal{A}}(x)$ and $w\in\mathsf{Z}_{\mathcal{A}}(x^{-1})$. Then $z+w\in\mathsf{Z}_{\mathcal{A}}(1)=\{ (0,\ldots,  0)\}$, and hence $0=|z+w|=|z|+|w|$. This is possible only if $z=(0,\ldots,  0)$, and it shows that $x=1$.

On the other hand, $H$ is finitely generated (by definition). Hence, $H$ is Krull if and only if it is root closed (see \cite[Theorem 2.7.14]{Ge-HK06a}). Accordingly, let $x\in \mathsf{q}(H)=\mathsf{q}(F)$ and $t\in\N$ such that $x^t\in H$. Since $\gamma_1(x^t)=t\gamma_1(x)$ and $\gamma'_i(x^t)=t\gamma'_i(x)$,  it follows by \eqref{E:H} that $x\in H$, and we are done.

\smallskip
\noindent
{\bf 3. Sets of lengths}. Let $x\in H$. We set
\[
C(x)=\sum_{i=1}^s\sum_{j=1}^{m_i}\gamma_{i,j}(x)+\sum_{i=2}^s(-m_1+m_i+1) \gamma'_i(x)
\quad\text{and}\quad
k(x)=\gamma_1(x)-\sum_{i=2}^s\gamma'_i(x)\fixed[0.35]{\text{ }}.
\]
We get from \eqref{E:H} that $k(x) \in \mathbb N_0$. We claim that
\begin{equation}\label{E:LSet}
\mathsf{L}_{\mathcal{A}}(x)=C(x)+\left\{\sum_{i=2}^s(-m_1+m_i+1)\beta_i: \beta_2,\ldots, \beta_s\in\N_0 \ \text{and}\  \sum_{i=2}^s\beta_i\leq k(x)\right\}.
\end{equation}
Indeed, we have by the above discussion (see, in particular, Step 1) that
\[
\begin{split}
B(x)&=\left\{(\beta_2,\ldots, \beta_s)\in\N_0^{s-1}: \beta_2\geq \gamma'_2(x), \ldots, \beta_s\geq \gamma'_s(x), \ \text{and}\ \sum_{i=2}^s\beta_i\leq \gamma_1(x)\right\}\\
&=\fixed[-0.25]{\text{ }}\bigl(\gamma'_2(x), \ldots, \gamma'_s(x)\bigr)+\left\{(\beta_2,\ldots, \beta_s)\in\N_0^{s-1}: \sum_{i=2}^{s}\beta_i\leq k(x)\right\}.
\end{split}
\]
Then, let $(\beta_2,\ldots,  \beta_s)\in\N_0^{s-1}$ be such that $\sum_{i=2}^s\beta_i\leq k(x)$, and set $\alpha=\mathsf{Z}_x(\gamma'_2(x)+\beta_2,\ldots,  \gamma'_s(x)+\beta_s)$. Plugging in the definitions, we obtain
\[
\begin{split}
|\alpha|
&= \sum_{j=1}^{m_1}\left(\gamma_{1,j}(x)-\sum_{i=2}^s(\gamma'_i(x)+\beta_i)\right)+
    \sum_{i=2}^s\sum_{j=1}^{m_i}\left(\gamma_{i,j}(x)+\gamma'_i(x)+\beta_i\right)+
    \sum_{i=2}^s(\gamma'_i(x)+\beta_i) \\
&=\sum_{i=1}^s\sum_{j=1}^{m_i}\gamma_{i,j}(x)+\sum_{i=2}^s(-m_1+m_i+1)\gamma'_i(x)
    +\sum_{i=2}^s(-m_1+m_i+1)\beta_i \\
&=C(x)+\sum_{i=2}^s(-m_1+m_i+1)\beta_i\fixed[0.35]{\text{ }}.
\end{split}
\]
Therefore, we obtain from \eqref{equ:rewriting-LA(x)} that $\mathsf{L}_{\mathcal{A}}(x)$ has the form given in \eqref{E:LSet}.

\smallskip
\noindent
{\bf 4. Claim:  $\mathcal{A}=\mathcal{A}(H)$}.

Indeed, since $H$ is reduced (by Step 2) and $H=[\mathcal{A}]$, we have $\mathcal{A}(H)\subseteq \mathcal{A}$. Conversely, let $u\in\mathcal{A}$. We want to show that $\mathsf{L}_{\mathcal{A}}(u)=\{ 1\}$, which  implies that $u\in\mathcal{A}(H)$.

To this end, assume first that $u=u_{k,l}$ for some $k\in \llb 1,s \rrb$ and $l\in \llb 1,m_k \rrb$. Then, $\gamma_i(u)=\gamma_i^\prime(u)=0$ for all $i\in \llb 1,s \rrb$, because $m_i \geq 2$. Thus, $C(u)=1$ and $k(u)=0$, and hence $\mathsf{L}_{\mathcal{A}}(u)=\{1\}$ by \eqref{E:LSet}.

Next, suppose that $u=u_{k,0}$ for some $k\in \llb 2,s \rrb$. Then we have
\[
\gamma_i(u)=\begin{cases}
1 & \text{if } i=1\\
-1 & \text{if } i=k\\
0 & \text{if } i\neq  1,k
\end{cases}\fixed[0.35]{\text{ }}.
\]
Thus, we find that $\gamma'_k(u)=1$, $\gamma'_i(u)=0$ if $i \ne k$, and $k(u)=0$. Consequently, we have by \eqref{E:LSet} that
\[
\mathsf{L}_{\mathcal{A}}(u)=\{ C(u)\}=\{ m_1-m_k+(-m_1+m_k+1)\}=\{ 1\}\fixed[0.35]{\text{ }}.
\]

\smallskip
\noindent
{\bf 5. Minima and maxima of sets of lengths}.
Let $x \in H$.
Since $\mathcal{A}=\mathcal{A}(H)$ by Step 4, we have
\begin{equation}\label{E:LSet2}
\mathsf L_H (x) = \mathsf{L}_{\mathcal{A}}(x)=C(x)+\left\{\sum_{i=2}^s(-m_1+m_i+1)\beta_i \colon \beta_2,\ldots, \beta_s\in\N_0 \ \text{and}\ \sum_{i=2}^s\beta_i\leq k(x)\right\},
\end{equation}
and we assert that
\begin{equation}
\label{equ:min-max}
\min\mathsf{L}_H(x)=C(x), \quad \max\mathsf{L}_H(x)=C(x)+k(x)(-m_1+m_s+1), \quad \text{and} \quad  C(x)\geq k(x)m_1 \fixed[0.35]{\text{ }}.
\end{equation}
Indeed, the first equality follows immediately from \eqref{E:LSet2}. As for the second, it is clear that
\[
C(x)+k(x)(-m_1+m_s+1)\in \mathsf L_H(x)\fixed[0.35]{\text{ }},
\]
because $0+\ldots + 0+k(x)\leq k(x)$. Conversely, if $\beta_2,\ldots, \beta_s\in \N_0$ and $\sum_{i=2}^{s}\beta_i\leq k(x)$, then
\[
C(x)+\sum_{i=2}^{s} (-m_1+m_i+1)\beta_i \leq C(x)+(-m_1+m_s+1)\sum_{i=2}^s\beta_i\leq C(x)+k(x)(-m_1+m_s+1)\fixed[0.35]{\text{ }},
\]
where we have used that $m_1 \le \ldots \le m_s$ (by construction). Therefore, $\max\mathsf{L}_H(x)$ is as stated.

Finally (recall that $\gamma'_i(x)=\max\{0,-\gamma_i(x)\}$, and hence $\gamma_i(x)\geq -\gamma'_i(x)$), we have
\[
\begin{split}
C(x)&= \sum_{i=1}^s\sum_{j=1}^{m_i}\gamma_{i,j}(x)+\sum_{i=2}^s\gamma'_i(x)(-m_1+m_i+1)\geq
\sum_{i=1}^sm_i\gamma_i(x)+\sum_{i=2}^s\gamma'_i(x)(-m_1+m_i+1) \\
&\geq m_1\gamma_1(x)-\sum_{i=2}^s\gamma'_i(x)m_i+\sum_{i=2}^s\gamma'_i(x)(-m_1+m_i+1)=
m_1\gamma_1(x)-(m_1-1)\sum_{i=2}^s \gamma'_i(x) \\
&\geq m_1 \fixed[-0.75]{\text{ }}\left(\gamma_1(x)-\sum_{i=2}^s\gamma'_i(x)\right)\fixed[-0.75]{\text{ }}=m_1k(x)\fixed[0.35]{\text{ }}.
\end{split}
\]

\smallskip
\noindent
{\bf 6. Claim:} $L \in \mathscr{L} (H)$ and $\UUc_{\min L} (H)=L$.

Note that $\min L = m_1$, and set $x=u_{1,1} \cdot \ldots \cdot u_{1,m_1}$. Then $\gamma_1(x)=1$, $\gamma'_i(x)=0$ for each  $i \in \llb 2, s \rrb$, $C(x)=m_1$, and $k(x)=1$. So we have by \eqref{E:LSet2} that
\[
\mathsf{L}_H(x)=m_1+\{0,-m_1+m_2+1,\ldots,  -m_1+m_s+1\}=L\fixed[0.25]{\text{ }}.
\]
Since $m_1 \in \mathsf{L}_H(x)=L$,  we obtain $L\subseteq \mathscr{U}_{m_1}(H)$.

Conversely, let $x\in H$ such that $m_1\in \mathsf{L}_H(x)$. We need to show that $\mathsf{L}_H(x)\subseteq L$. Since $m_1\in\mathsf{L}_H(x)$, we get from \eqref{equ:min-max} and \eqref{E:LSet2} that
\begin{equation*}
\label{equ:uplow-bounds}
k(x)m_1\leq C(x)=\min\mathsf{L}_H(x)\leq m_1\fixed[0.35]{\text{ }}.
\end{equation*}
It follows that $k(x)=0$ or $k(x)=1$, and this concludes the proof of the claim, as we have by \eqref{E:LSet2} and the above that $\mathsf{L}_H(x)=\{C(x)\} = \{m_1\} \subseteq L$ if $k(x) = 0$, and $L_H(x) = L$ otherwise.

\smallskip
\noindent
{\bf 7. Claim: $\max \Delta (H) = \max \Delta (L)$}.

Since $L \in \mathscr{L} (H)$ by the previous step, it is enough to prove that $\max \Delta (H) \le  \max \Delta (L)$.
Let $x\in H$ such that $\Delta(\mathsf{L}_H(x))$ is non-empty, and
let $l<l'$ be consecutive elements of $\mathsf{L}_H(x)$. We have to show that $l'-l\leq \max\Delta(L)$. To this end,
recall \eqref{E:LSet2} and pick $\beta_2,\ldots,\beta_s \in \N_0$ such that $\sum_{i=2}^s\beta_i\leq k(x)$ and $l=C(x)+\sum_{i=2}^s(-m_1+m_i+1)\beta_i$. We distinguish two cases.

\smallskip
\noindent \textsc{Case 1:} $\beta_{i_0}>0$ for some $i_0 \in \llb 2, s-1 \rrb$.

For every $i \in \llb 2, s \rrb$, we define $\beta'=(\beta'_2,\ldots,\beta'_s)$ by
\[
\beta'_i=\begin{cases} \beta_i & \text{if } i\neq  i_0, i_0+1, \\
\beta_{i_0}-1 & \text{if } i=i_0, \\
\beta_{i_0+1}+1 & \text{if } i=i_0+1.
\end{cases}
\]
Then $\beta'\in\N_0^{s-1}$ and $\sum_{i=2}^s\beta'_i=\sum_{i=2}^s\beta_i\leq k(x)$, so we obtain from \eqref{E:LSet2} that
\[
\begin{split}
\mathsf{L}_H(x)\ni l'' & :=C(x)+\sum_{i=2}^s(-m_1+m_i+1)\beta'_i \\
&=C(x)+\sum_{i=2}^s(-m_1+m_i+1)\beta_i-(-m_1+m_{i_0}+1)+(-m_1+m_{i_0+1}+1) \\
&=l+(m_{i_0+1}+1)-(m_{i_0}+1).
\end{split}
\]
Using that $m_i < m_{i+1}$ for every $i \in \llb 2, s-1 \rrb$, it follows that $l<l'\leq l''$, and therefore
\[
l'-l\leq l''-l=(m_{i_0+1}+1)-(m_{i_0}+1)\leq \max\Delta(L)\fixed[0.35]{\text{ }}.
\]

\smallskip
\noindent \textsc{Case 2:} $\beta_i=0$ for each  $i \in \llb 2, s-1 \rrb$.

Then
$l=C(x)+(-m_1+m_s+1)\beta_s$, and we have $\beta_s<k(x)$, since we get from \eqref{equ:min-max} that
\[
l<\max \mathsf{L}_H(x)=C(x)+k(x)(-m_1+m_s+1)\fixed[0.35]{\text{ }}.
\]
It follows that
$1+\beta_s\leq k(x)$, and then we infer from \eqref{E:LSet2} that
\[
\mathsf{L}_H(x)\ni l'':=C(x)+(-m_1+m_2+1)+(-m_1+m_s+1)\beta_s=l+(-m_1+m_2+1)>l\fixed[0.35]{\text{ }}.
\]
So $l<l'\leq l''$ and $l'-l\leq l''-l=(m_2+1)-m_1 \leq \max\Delta(L)$.

\smallskip
\noindent
{\bf 8. Claim:} $\rho_{\ell \min L + \nu} (H) = \ell \max L + \nu$ for each $\nu \in \llb 0, \min L - 1 \rrb$ and all $\ell \in \N_0$.

Clearly, we have $\min L = m_1$ and $\max L = m_s+1$.
Let $\ell\in\N_0$ and $\nu\in \llb 0,m_1-1 \rrb$.  From $u_{s,0}=(u_{1,1} \cdot \ldots \cdot u_{1,m_1})(u_{s,1} \cdot \ldots \cdot u_{s,m_s})^{-1}$ we obtain
\[
(u_{1,1} \cdot \ldots \cdot  u_{1,m_1})^{\ell}(u_{1,1} \cdot \ldots \cdot  u_{1,\nu})=(u_{s,0}u_{s,1} \cdot \ldots \cdot u_{s,m_s})^{\ell}(u_{1,1} \cdot \ldots \cdot u_{1,\nu})\fixed[0.35]{\text{ }}.
\]
Since $\mathcal A(H) = \mathcal A$ by Step 4, we thus find that $\rho_{\ell m_1+\nu}(H) \geq \ell (m_s+1)+\nu$.

Conversely, let $x\in H$ be such that $\ell m_1+\nu\in  \mathsf{L}_H(x)$. It follows from Step 5 that
\[
\ell m_1+\nu \geq \min \mathsf{L}_H(x)=C(x)\geq k(x)m_1\fixed[0.35]{\text{ }},
\]
which implies $k(x)\leq \ell$. Therefore, we have by \eqref{equ:min-max} that
\[
\max \mathsf{L}_H(x)=C(x)+k(x)(-m_1+m_s+1)\leq \ell m_1+\nu+\ell(-m_1+m_s+1)=\ell(m_s+1) +\nu\fixed[0.35]{\text{ }},
\]
which, in turn, implies $\rho_{\ell m_1+\nu}(H) \leq \ell (m_s+1)+\nu$.
\end{proof}

\medskip
\begin{theorem} \label{4.2}
 Let $d \in \N_{\ge 2}$ and let  $(m_k)_{k \ge 0}$ be a sequence of non-negative integers such that  $m_0=0$, $m_1=1$,  and
\begin{equation}
 \sum_{i=1}^{k-2} \left( \frac{k}{i} m_i + (i-1) \right)\fixed[-0.8]{\text{ }}  \le m_{k-1} - (k-1) \quad \text{for each} \quad k \ge 3 \fixed[0.25]{\text{ }}.
\end{equation}
Moreover, let $(U_k)_{k \ge 1}$ be a family of finite subsets of the positive integers with $U_1=\emptyset$ such that, for all $k \ge 2$, $U_k$ is non-empty, $m_{k-1}+1 < \min U_k \le m_{k-1}+d$, $\max U_k = m_k$, and $\max \Delta (U_k)\le d$. Then there exists a locally tame Krull monoid $H$ with finite set of distances $\Delta (H) \subseteq \llb 1,d \fixed[0.25]{\text{ }} \rrb$ and with
\[
\UUc_k (H) \cap \N_{\ge k} =  \llb k, m_{k-1}+1 \rrb \uplus U_k \quad \text{for each} \quad k \in \N\fixed[0.25]{\text{ }}.
\]
In particular, if  $\min U_k + 1 \notin U_k$ for all $k \ge 2$, then there exists no $M \in \N$ such that $\UUc_k (H)$ is an {\rm AAP} with difference $1$ and bound $M$ for all sufficiently large $k$, and hence $\mathscr L (H)$ does not satisfy the Structure Theorem for Unions.
\end{theorem}

\smallskip
\noindent
\textit{Comment.} Of course, we can combine the present result, showing the existence of a Krull monoid $H$ with prescribed properties for $\LLc (H)$, with Claborn's Realization Theorem for class groups (\cite[Theorem 3.7.8]{Ge-HK06a}) or with one of the realization theorems for monoids of modules (mentioned in Subsection \ref{subsec:3.3}). Then we obtain a Dedekind domain $R$, respectively a monoid of modules $\mathcal V ( \mathcal C)$,  such that $\LLc (R)$, respectively $\LLc \bigl( \mathcal V (\mathcal C) \bigr)$, has the prescribed properties.

\begin{proof}
The ``In particular'' part is a simple consequence of the main statement.

As for the rest, let $d$, $(m_k)_{k \ge 0}$, and $(U_k)_{k \ge 1}$ be given with the above properties. For each $k \in \N$ we set
\begin{equation*}
\label{equ:defining-Lk}
L_k =  \llb k, m_{k-1}+1 \rrb \uplus U_k \fixed[0.35]{\text{ }},
\end{equation*}
and we note that $L_1 = \{1\}$. In addition, we set $H_1 = (\N_0, +)$, and for every $k \ge 2$ we choose
a finitely generated reduced Krull monoid $H_k$ with $L_k \in \LLc (H_k)$ that satisfies all the conditions specified by points \ref{4.1(i)}-\ref{4.1(iii)} of Proposition \ref{4.1}. We claim that the coproduct
\[
H = \coprod_{k \ge 1} H_k
\]
has all the required properties. Because finitely generated monoids are locally tame (\cite[Theorem 3.1.4]{Ge-HK06a}) and the coproduct of locally tame Krull monoids is a locally tame Krull monoid (\cite[Proposition 1.6.8]{Ge-HK06a}), the monoid $H$ is a locally tame Krull monoid. Furthermore, since $\sup \Delta (H) = \sup \{\sup \Delta  (H_k) : k \in \N \}$ by \cite[Proposition 1.4.5]{Ge-HK06a}, it follows that $\Delta  (H) \subseteq \llb 1,  d \fixed[0.25]{\text{ }} \rrb$.
So we are left with the statement on the unions $\UUc_k (H)$.
For this, we will prove the following four assertions, the last of which implies the theorem.

\begin{enumerate}
\item[({\bf A1})\,] For all $k \in \N$, we have $\UUc_k (H_k) = L_k$ and $\UUc_{\nu} (H_k) = \{\nu\}$ for $\nu \in \llb 1, k-1 \rrb$.

\smallskip

\item[({\bf A2})\,] $\rho_k (H_1 \times \ldots \times H_{k-1}) = 1 + m_{k-1}$ for each $k \in \N_{\ge 2}$.

\smallskip

\item[({\bf A3})\,] $\UUc_k (H_1 \times \ldots \times H_k) \cap \N_{\ge k} = L_k$ for every $k \in \N$.

\smallskip

\item[({\bf A4})\,] $\UUc_k (H) \cap \N_{\ge k} =L_k$ for all $k \in \N$.
\end{enumerate}

\smallskip
To start with, we recall that, for every atomic monoid $S$ with $S \ne S^{\times}$, we have $\UUc_1 (S) = \{1\}$,  $\UUc_0 (S) = \{0\}$, and $\rho_0 (S)=0$. Note that, by construction, the monoids $H_k$  are reduced, and hence $H_k \ne H_k^{\times} = \{1\}$, for all $k \in \N$.

\smallskip
\textit{Proof of \textup{(}{\bf A1}\textup{)}}.\, This is straightforward from Conditions \ref{4.1(ii)} and \ref{4.1(iii)} of Proposition \ref{4.1} (and the fact that, by construction, $\min L_k = k$). \qed ({\bf A1})

\smallskip
\textit{Proof of \textup{(}{\bf A2}\textup{)}}.\, The assertion holds for $k=2$, because $\rho_2 (H_1)=2$ and $m_1 = 1$. So, let $k \ge 3$ and set $S = H_1 \times \ldots \times H_{k-2}$. We want to show that
\[
\rho_k (S \times H_{k-1}) = \max \fixed[-0.25]{\text{ }}\bigl\{ \rho_{\nu} (S) + \rho_{k-\nu}(H_{k-1}) : \nu \in \llb 0,k \rrb \bigr\} = 1 + m_{k-1} .
\]
To begin, we have by ({\bf A1}) and Condition \ref{4.1(iii)} of Proposition \ref{4.1} that
$$
\rho_0 (S) + \rho_k (H_{k-1}) = 1 + m_{k-1} = \rho_1 (S) + \rho_{k-1}(H_{k-1})
$$
and $\rho_{k-\nu} (H_{k-1}) =k-\nu$ for every $\nu \in \llb 2,k \rrb$. Consequently, it is enough to prove that
\[
\rho_{\nu} (S) \le m_{k-1}-(k-1)+\nu \quad \text{for each} \quad \nu \in \llb 2,k \rrb.
\]
To this end, let $i \in \llb 1, k-2 \rrb$ and write $k=\ell i + \nu$, where $\ell \in \N_0$ and $\nu \in \llb 0,i-1 \rrb$. Then, appealing again to Condition \ref{4.1(iii)} of Proposition \ref{4.1}, we find that
\[
\rho_k (H_i) = \rho_{\ell i + \nu} (H_i) = \ell m_i + \nu \le \frac{k}{i} m_i + (i-1) .
\]
For each $\nu \in \llb 2,k \rrb$, we thus obtain that
\[
\pushQED{\qed}
\rho_{\nu}(S)\le \rho_k (S) \le \sum_{i=1}^{k-2} \rho_k (H_i) \le \sum_{i=1}^{k-2} \left( \frac{k}{i} m_i + (i-1) \right) \fixed[-0.75]{\text{ }}\le m_{k-1} - (k-1) . \qedhere ({\bf A2})
\popQED
\]

\smallskip
\textit{Proof of \textup{(}{\bf A3}\textup{)}}.\, The statement holds for $k=1$. So, assume $k \ge 2$ and set $S = H_1 \times \ldots \times H_{k-1}$. Then
\[
\begin{aligned}
\UUc_k (S \times H_k) \cap \N_{\ge k} & = \bigcup_{\nu=0}^k \big( \UUc_{\nu} (S) + \UUc_{k-\nu}(H_k) \big) \cap \N_{\ge k} \\
 & = \big( \UUc_k (H_k) \cap \N_{\ge k} \big) \cup \left(  \bigcup_{\nu=1}^k \big( \UUc_{\nu} (S) + \UUc_{k-\nu}(H_k) \big) \cap \N_{\ge k}  \right) \\
 & = \big( \UUc_k (H_k) \cap \N_{\ge k} \big) \cup \left(  \bigcup_{\nu=1}^k \big( (k-\nu) + \UUc_{\nu} (S)  \big) \cap \N_{\ge k}  \right)\fixed[-0.35]{\text{ }}.
 \end{aligned}
\]
Since for each $\nu\in \llb 1,k \rrb$ we have $(k-\nu) + \UUc_{\nu} (S) \subseteq \UUc_k (S)$,
we see that
\[
\begin{aligned}
L_k \subseteq \UUc_k (S \times H_k) \cap \N_{\ge k} & \subseteq \bigl( \UUc_k (H_k) \cap \N_{\ge k} \bigr) \cup \bigl( \UUc_k (S) \cap \N_{\ge k} \bigr) = L_k \cup \bigl( \UUc_k (S) \cap \N_{\ge k} \bigr) = L_k \,,
\end{aligned}
\]
where the last equality follows from ({\bf A2}).
\qed ({\bf A3})

\smallskip
\textit{Proof of \textup{(}{\bf A4}\textup{)}}.\, The statement is clear for $k=1$. So, suppose that $k \ge 2$, and set $S_1 = H_1 \times \ldots \times H_k$ and $S_2 = \coprod_{\nu \ge k+1} H_{\nu}$, whence $H = S_1 \times S_2$. If $\nu \in \llb 0,k \rrb$, then ({\bf A1}) implies $\UUc_{\nu} (S_2) = \{\nu\}$. Therefore,
\[
\begin{aligned}
\UUc_k (S_1 \times S_2) \cap \N_{\ge k} & = \bigcup_{\nu=0}^k \big( \UUc_{k-\nu} (S_1) + \UUc_{\nu} (S_2) \big) \cap \N_{\ge k} \\
 & = \bigcup_{\nu=0}^k \big( \nu + \UUc_{k-\nu} (S_1)  \big) \cap \N_{\ge k} \\
 & = \bigl( \UUc_k (S_1) \cap \N_{\ge k} \bigr) \cup \left( \bigcup_{\nu=1}^k \big( \nu + \UUc_{k-\nu} (S_1)  \big) \cap \N_{\ge k} \right)\fixed[-0.35]{\text{ }},
\end{aligned}
\]
which, together with Lemma \ref{prop:basic(1)}\ref{it:prop:basic(1)(ii)}, shows that, for every $\nu \in \llb 1,k \rrb$,
\[
\nu + \UUc_{k-\nu}(S_1) \subseteq \UUc_{\nu} (S_1) + \UUc_{k-\nu} (S_1) \subseteq \UUc_k (S_1) \,.
\]
So, putting it all together, we get from ({\bf A3}) that
\[
\pushQED{\qed}
\UUc_k (S_1 \times S_2) \cap \N_{\ge k} = \UUc_k (S_1) \cap \N_{\ge k} = L_k. \qedhere ({\bf A4})
\popQED
\]
This completes the proof of the theorem.
\end{proof}

\providecommand{\bysame}{\leavevmode\hbox to3em{\hrulefill}\thinspace}
\providecommand{\MR}{\relax\ifhmode\unskip\space\fi MR }
\providecommand{\MRhref}[2]{%
  \href{http://www.ams.org/mathscinet-getitem?mr=#1}{#2}
}
\providecommand{\href}[2]{#2}

\end{document}